\documentclass[12pt]{amsart}
\usepackage{latexsym,amssymb,amsmath}
\usepackage{latexsym,amssymb,amsmath,tikz,verbatim,cancel,hyperref}

\usepackage{multicol}
\usepackage[normalem]{ulem}

\numberwithin{equation}{section}
\newtheorem{theorem}{Theorem}[section]
\newtheorem{lemma}[theorem]{Lemma}
\newtheorem{proposition}[theorem]{Proposition}
\newtheorem{corollary}[theorem]{Corollary}
\newtheorem{conjecture}[theorem]{Conjecture}

\theoremstyle{definition}
\newtheorem{definition}[theorem]{Definition}

\newtheorem{remark}[theorem]{Remark}
\newtheorem{example}[theorem]{Example}

\newcommand{\init}{\mathrm{in}}

\begin{document}

\title[Geometric vertex decomposition and liaison for toric ideals of graphs]{Geometric vertex decomposition and liaison for toric ideals of graphs}

\thanks{\color{blue}The final version of this paper appeared in Algebraic Combinatorics, Vol. 6 (4), 2023, pp. 965--997. This version of the paper corrects a small error in the proof of Theorem 3.11.  The statement of the theorem remains unchanged, but the  corrected proof includes an additional case that was missing in the original version. To highlight the changes from the journal version, all changes are in BLUE.
}

\author[M. Cummings]{Mike Cummings}
\address[Mike Cummings]{Department of Mathematics and Statistics\\
McMaster University, Hamilton, ON L8S 4K1, Canada}
\email{cummim5@mcmaster.ca} 

\author[S. Da Silva]{Sergio Da Silva }
\address[Sergio Da Silva]{Deptartment of Mathematics and Economics, Virginia State University, 1 Hayden Drive, Petersburg, VA 23806, USA}
\email{sdasilva@vsu.edu }

\author[J. Rajchgot]{Jenna Rajchgot}
\address[Jenna Rajchgot]{Department of Mathematics and Statistics\\
McMaster University, Hamilton, ON L8S 4K1, Canada}
\email{rajchgot@math.mcmaster.ca }

\author[A. Van Tuyl]{Adam Van Tuyl}
\address[Adam Van Tuyl]{Department of Mathematics and Statistics\\
McMaster University, Hamilton, ON L8S 4K1, Canada}
\email{vantuyl@math.mcmaster.ca}

\date{\today}

\begin{abstract} 
The geometric vertex decomposability property for polynomial ideals is an ideal-theoretic generalization of the vertex decomposability property for simplicial complexes. 
Indeed, a homogeneous geometrically vertex decomposable ideal is radical and Cohen-Macaulay, and is in the Gorenstein liaison class of a complete intersection (glicci).

In this paper, we initiate an investigation into
when the toric ideal $I_G$ of a finite simple graph $G$ is geometrically vertex decomposable.  
We first show  
how geometric vertex
decomposability behaves under tensor products, which allows us to restrict to connected graphs.  We then 
describe a graph operation that preserves geometric vertex decomposability, thus allowing us to build
many graphs whose corresponding toric ideals are geometrically
vertex decomposable.
Using work of 
Constantinescu and Gorla, we prove that toric ideals
of bipartite graphs are geometrically vertex decomposable.
We also
propose a conjecture that all toric ideals of graphs
with a square-free degeneration with respect to a lexicographic
order are geometrically vertex decomposable.  As evidence, we prove the conjecture in the case that the
universal Gr\"obner basis of $I_G$ is a set of quadratic
binomials.  We also prove that some other families of
graphs have the property that $I_G$ is glicci.
\end{abstract}

\maketitle

\section{Introduction}

Vertex decomposable simplicial complexes are recursively defined
simplicial complexes that have been extensively studied in both
combinatorial algebraic topology and combinatorial commutative 
algebra. This family of complexes, first defined by
Provan and Billera \cite{PB} for pure simplicial
complexes and later generalized to the non-pure case by 
Bj\"orner and Wachs \cite{BW1997}, has many nice features.   
For example, they are shellable and hence Cohen-Macaulay in the pure case. 

Because of the Stanley-Reisner correspondence between 
square-free monomial ideals and simplicial complexes, the 
definition and properties of vertex decomposable simplicial complexes can be translated into 
algebraic statements about square-free monomial ideals.  For example, Moradi and Khosh-Ahang \cite[Definition 2.1]{MKA} introduced vertex splittable ideals, which are precisely the ideals of the Alexander duals of vertex decomposable simplicial complexes.
As another example, which is directly relevant to this paper, Nagel and R\"omer \cite{NR} 
showed that 
if $I_\Delta$ is the square-free monomial
ideal associated to a vertex decomposable simplicial complex $\Delta$
via the Stanley-Reisner correspondence, then the ideal 
$I_\Delta$ belongs 
to the Gorenstein liasion class of a complete intersection, i.e., 
the ideal $I_\Delta$ is {\it glicci}. 

Knutson, Miller, and Yong \cite{KMY} introduced the notion
of a {\it geometric vertex decomposition}, which is an ideal-theoretic generalization (beyond the square-free monomial ideal setting) of a vertex
decomposition of a simplicial complex.
Building on this, Klein and Rajchgot \cite{KR} gave a recursive definition of a {\it geometrically vertex decomposable} ideal which is an ideal-theoretic generalization of a vertex decomposable simplicial complex.  Indeed, when 
specialized to square-free monomial ideals, those ideals that are geometrically vertex decomposable are precisely those square-free monomial ideals whose associated simplicial complexes are vertex decomposable.  As shown by Klein and Rajchgot \cite[Theorem 4.4]{KR}, this definition captures some of the properties of
vertex decomposable simplicial complexes.  For example,
a more general version of Nagel and R\"omer's result holds; 
that is, a homogeneous 
ideal that is geometrically vertex decomposable is also glicci.
Because geometrically vertex decomposable ideals are glicci, identifying such families allows us
to give further evidence to an important open question
in liaison theory: is every arithmetically Cohen-Macaulay
subscheme of $\mathbb{P}^n$ glicci (see
\cite[Question 1.6]{KMMNP})?

Since the definition of geometrically vertex decomposable ideals is 
recent, there is a need to not only develop the corresponding theory
(e.g. which properties of Stanley-Reisner ideals of vertex decomposable simplicial complexes
also hold for geometrically vertex decomposable ideals?), but also a need to find families of concrete examples.    There has 
already been some work in these two directions.
Klein and Rajchgot \cite{KR} showed that
Schubert determinantal ideals, (homogeneous) ideals
coming from lower bound cluster algebras, and ideals defining equioriented 
type A quiver loci are all geometrically vertex decomposable.  Klein \cite{K}
used geometric vertex decomposability to
prove a conjecture of 
Hamaker, Pechenik, and Weigandt \cite{HPW} on Gr\"obner bases 
of Schubert determinantal ideals. Da Silva 
and Harada have investigated the 
geometric vertex decomposability of certain Hessenberg patch ideals which locally define regular nilpotent Hessenberg varieties \cite{DSH}.  

We contribute to this program by
further developing the theory of geometric
vertex decomposibility, and show that many families
of toric ideals of graphs are geometrically vertex
decomposable.  
Let $\mathbb{K}$ be an algebraically closed field of characteristic $0$.
If
$G = (V,E)$ is a finite simple graph
with vertex set $V = \{x_1,\ldots,x_m\}$ and edge
set $E = \{e_1,\ldots,e_n\}$,  we can define
a ring homomorphism $\varphi:\mathbb{K}[e_1,\ldots,e_n] \rightarrow \mathbb{K}[x_1,\ldots,x_m]$ by letting
$\varphi(e_i) = x_kx_l$ where the edge
$e_i = \{x_k,x_l\}$.  The {\it toric ideal of $G$}
is the ideal $I_G = {\rm ker}(\varphi)$. 
The study of toric ideals of graphs is an active
area of research (e.g. see \cite{BOVT,CG,FHKVT,GHKKPVT,GM,OH,TT,V1}), so our
work also complements the recent developments in this area.  What makes toric ideals of graphs amenable to our investigation of  geometric vertex decomposability is that their (universal) Gr\"obner bases are fairly well-understood
(see Theorem \ref{generatordescription}) and can be related to the graph's structure.

Our first main result describes how
geometric vertex decomposability behaves over tensor products:

\begin{theorem}[Theorem \ref{tensorproduct}]
 Let $I \subsetneq R =\mathbb{K}[x_1,\ldots,x_n]$ and 
 $J \subsetneq S=\mathbb{K}[y_1,\ldots,y_m]$ be proper ideals. Then $I$ and $J$ are geometrically
 vertex decomposable if and only if $I+J$ is geometrically vertex decomposable in $R \otimes S =\mathbb{K}[x_1,\ldots,x_n,y_1,\ldots,y_m]$.
\end{theorem}

\noindent
Our result can be viewed as the ideal-theoretic version
of the fact that two simplicial complexes 
$\Delta_1$ and $\Delta_2$ are vertex decomposable 
if and only if
their join $\Delta_1 \star \Delta_2$ is vertex decomposable \cite[Propostion 2.4]{PB}.  Moreover, this result allows
us to reduce our study of toric ideals of graphs
to the case that the graph $G$ is connected (Theorem \ref{connected}).

When we restrict to toric ideals of graphs, we
show that the graph operation of ``gluing'' an even length cycle
onto a graph preserves the geometric vertex decomposability property:

\begin{theorem} [Theorem \ref{gluetheorem}]
Let $G$ be a finite simple graph 
with toric ideal $I_G$.
Let $H$ be obtained from $G$ by gluing a cycle of even length to $G$ along a single edge. If $I_G$ is geometrically vertex decomposable, then $I_H$ is also geometrically vertex decomposable.
\end{theorem}

This gluing operation and its connection to toric ideals of graphs appears in work of Favacchio, Hofscheier, Keiper and Van Tuyl \cite{FHKVT}, while a similar construction
of using $H$-paths is employed by
Gitler, Reyes, and Villarreal \cite{GRV}
to characterize the toric ideals of bipartite 
graphs that are complete intersections.
By repeatedly applying
this operation, we can construct many toric
ideals of graphs that are geometrically vertex decomposable and glicci. 

Our gluing operation requires one to start with a graph
whose corresponding toric ideal is geometrically vertex decomposable.  It
is therefore desirable to identify families of graphs whose
toric ideals have this property.  Towards this end, we prove:

\begin{theorem}[Theorem \ref{thm: gvdBipartite}]\label{mainresult3}
Let $G$ be a finite simple graph with toric ideal $I_G$. If $G$ is bipartite, then $I_G$ is geometrically vertex decomposable.
\end{theorem}

\noindent
Our proof of Theorem \ref{mainresult3} relies on work of Constantinescu and Gorla \cite{CG}.  For some families of
bipartite graphs, we give alternative proofs for the geometric
vertex decomposable property that exploit
the additional structure of the graph (see Theorem
\ref{families}).  These families
are also used to illustrate that in certain cases, the recursive definition of geometric vertex decomposability  easily lends itself to induction. 

Based on our results and computer experimentation in Macaulay2 \cite{M2}, we propose
the following conjecture:

\begin{conjecture}[Conjecture \ref{mainconjecture}]\label{conjecture}
Let $G$ be a finite simple graph with toric ideal $I_G \subseteq \mathbb{K}[e_1,\ldots,e_n]$.
If $\init_{<}(I_G)$ is square-free with respect to a lexicographic monomial order $<$, then $I_G$ is geometrically vertex decomposable, and thus glicci.
\end{conjecture}

\noindent 
We provide a framework to prove this conjecture.  In fact, we show
that the conjecture
is true if one can prove that a particular family
of ideals is equidimensional (see Theorem \ref{framework}).   As further evidence
for Conjecture \ref{conjecture}, we prove the following special case: 

\begin{theorem}[Theorem \ref{quadratic_GVD}]
Let $I_G$ be the toric ideal of a finite simple graph $G$. Assume that $I_G$ has a universal Gr\"obner basis consisting entirely of quadratic binomials. Then $I_G$ is geometrically vertex decomposable.
\end{theorem}

Finally, we prove that additional collections of toric ideals of graphs are glicci (though not necessarily geometrically vertex decomposable). Our first result in this direction relies on a very general result of Migliore and Nagel \cite[Lemma 2.1]{MN} from the liaison literature.

\begin{theorem}[Corollary \ref{glue=glicci}]
Let $G$ be a finite simple graph and let $I_G\subseteq R = \mathbb{K}[e_1,\dots, e_n]$ be its toric ideal. 
Let $H$ be obtained from $G$ by gluing a cycle of even length to $G$ along a single edge. If $R/I_G$ is Cohen-Macaulay, then $I_H$ is glicci. 
\end{theorem}

We also show that many toric ideals of graphs which contain $4$-cycles are glicci. The following is a slightly weaker version of Corollary \ref{cor:square-freeDegenGlicciGraph}. 

\begin{theorem}[Corollary \ref{cor:square-freeDegenGlicciGraph}]
Let $G$ be a finite simple graph and suppose there is an edge $y\in E(G)$ contained in a $4$-cycle. If the initial ideal $\init_< I_G$ is a square-free monomial ideal for some lexicographic monomial order with $y>e$ for all $e\in E(G)$ with $e\neq y$, then $I_G$ is glicci.
\end{theorem}

As a corollary to this theorem, we show that the toric ideal of any \emph{gap-free} graph which contains a $4$-cycle is glicci. For the definition of gap-free graph and this result, see the end of Section \ref{sec:someGlicci}. 
 
\noindent{\bf Outline of the paper.} In the next section
we formally introduce geometrically vertex decomposable ideals, along with the required background and notation about Gr\"obner bases.
We also explain how geometrically vertex decomposable ideals
behave with respect to tensor products.  In Section 3 we provide the
needed background on toric ideals of graphs, and we explain how
a particular graph operation preserves the geometric
vertex decomposability property.  In Section 4, we focus on
the glicci property for toric ideals of graphs that can 
be deduced from the results of Section 3 together with general results from the liaison theory literature.  In Section 5 we 
prove that toric ideals of bipartite graphs are geometrically vertex decomposable.  In Section 6
we propose a conjecture on toric ideals with a square-free initial
ideal, 
describe a framework to prove this conjecture, and illustrate this framework by proving that toric ideals of graphs which have quadratic universal Gr\"obner bases are geometrically vertex decomposable.

\noindent{\bf Remark on the field $\mathbb{K}$. }Many of the arguments in this paper are valid over any infinite field. Indeed, the liaison-theoretic setup in Sections 2 and 4 requires an infinite field but is characteristic-free. Similarly, toric ideals of graphs can be defined combinatorially, and since the coefficients of their generators are $\pm 1$, defining such ideals in positive characteristic does not pose any issues. Nevertheless, we assume that $\mathbb{K}$ throughout this paper is algebraically closed of characteristic zero since some of the references that we use make this assumption (e.g. \cite[Proposition 13.15]{Sturm}, which is needed in the proof of Theorem \ref{sqfree=>cm}).

\noindent
{\bf Acknowledgments.} 
We thank Patricia Klein for some helpful conversations.
Cummings was partially supported by an NSERC USRA. Da Silva was partially supported by an NSERC postdoctoral fellowship.
Rajchgot's research is supported
by NSERC Discovery Grant 2017-05732.
 Van Tuyl’s research is supported by NSERC Discovery Grant 2019-05412.

\section{Geometrically vertex decomposable ideals}
 
In this paper $\mathbb{K}$ denotes 
 an algebraically closed field of 
characteristic zero and  $R = \mathbb{K}[x_1,\ldots,x_n]$ 
is the polynomial ring
in $n$ variables.
This section gives the required background on geometrically
vertex decomposable ideals, following \cite{KR}. 
We also examine how geometric vertex
decomposability behaves over tensor products.
 
 Fix a variable $y = x_i$ in $R$.   For any
 $f \in R$, we can write $f$ as $f = \sum_i \alpha_iy^i$, where $\alpha_i$
 is a polynomial only in the variables 
 $\{x_1,\ldots,\hat{x}_i,\ldots,x_n\}$. For $f\neq 0$, the
 {\it initial $y$-form} of $f$, denoted ${\rm in}_y(f)$, is the
 non-zero coefficient of the highest power of $y$ appearing
 in $\sum_i \alpha_iy^i$. That is, 
 if $\alpha_d \neq 0$, but $\alpha_t =0$ for all $t > d$, then
 ${\rm in}_y(f) = \alpha_dy^d$.  Note that if $y$ does not appear in any
 term of $f$, then ${\rm in}_y(f) = f$.  For any ideal $I$ of $R$,
 we set ${\rm in}_y(I) = \langle {\rm in}_y(f) ~|~ f \in I \rangle$ to
 be the ideal generated by all the initial $y$-forms in $I$.  A monomial
 order $<$ on $R$ is said to be {\it $y$-compatible} if the initial
 term of $f$ satisfies
${\rm in}_<(f) = {\rm in}_<({\rm in}_y(f))$ for all $f \in R$.  For
such an order, we have ${\rm in}_<(I) = {\rm in}_<({\rm in}_y(I))$, where ${\rm in}_<(I)$ is
the {\it initial ideal} of $I$ with respect to the
order $<$.
 
Given an ideal $I$ and a $y$-compatible monomial order $<$, let
$\mathcal{G}(I) = \{ g_1,\ldots,g_m\}$ be a Gr\"obner basis of $I$
with respect to this monomial order.  For $i=1,\ldots,m$, write $g_i$ as
$g_i = y^{d_i}q_i + r_i$, where $y$ does not divide any term of $q_i$; that is, ${\rm in}_y(g_i) = y^{d_i}q_i$.  It can then be shown that
${\rm in}_y(I) = \langle y^{d_1}q_1,\ldots,y^{d_m}q_m \rangle$ (see 
\cite[Theorem 2.1(a)]{KMY}).
 
 Given this setup, we define two ideals:
 $$C_{y,I} = \langle q_1,\ldots,q_m \rangle ~~\mbox{and}~~
 N_{y,I} = \langle q_i ~|~ d_i = 0 \rangle.$$ 
Recall that an ideal $I$ is {\it unmixed} if the
ideal $I$ 
satisfies $\dim(R/I) = \dim(R/P)$ for
all associated primes $P \in {\rm Ass}_R(R/I)$.  We come to our
main definition:

\begin{definition}\label{gvd}
 An ideal $I$ of $R = \mathbb{K}[x_1,\ldots,x_n]$ is {\it geometrically
 vertex decomposable} if $I$ is unmixed and 
 \begin{enumerate}
     \item $I = \langle 1 \rangle$, or $I$ is generated by
     a (possibly empty) subset of variables of $R$, or 
     \item there is a variable $y = x_i$ in $R$ and a
     $y$-compatible monomial order $<$ such that 
     $${\rm in}_y(I) = C_{y,I} \cap (N_{y,I} + \langle y \rangle),$$
     and the contractions of the
     ideals $C_{y,I}$ and $N_{y,I}$ to the ring
     $\mathbb{K}[x_1,\ldots,\hat{x}_i,\ldots,x_n]$ are geometrically
     vertex decomposable.
 \end{enumerate}
We make the convention that the two ideals 
$\langle 0 \rangle$ and $\langle 1 \rangle$ of the ring 
$\mathbb{K}$ are also geometrically
vertex decomposable.
\end{definition}

\begin{remark}\label{gvdremark}
For any ideal $I$ of $R$, if there exists a variable $y = x_i$ in $R$ and a
     $y$-compatible monomial order $<$ such that 
     ${\rm in}_y(I) = C_{y,I} \cap (N_{y,I} + \langle y \rangle)$, then this
     decomposition is called a {\it geometric vertex decomposition of $I$ with
     respect to $y$}.  This decomposition was first defined in
     \cite{KMY}.
     Consequently, Definition \ref{gvd} (2) says that there
     is a variable $y$ such that $I$ has a geometric vertex decomposition with respect
     to this variable.
     
     We say that a geometric vertex decomposition is \textit{degenerate} if either $C_{y,I}=\langle 1\rangle$ or $\sqrt{C_{y,I}}=\sqrt{N_{y,I}}$ (see \cite[Section 2.2]{KR} for further details and results). Otherwise, we call a geometric vertex decomposition \textit{nondegenerate}.
     \end{remark}
 
 If elements in our Gr\"obner basis are
 square-free in $y$, i.e., if ${\rm in}_y(g_i) = y^{d_i}q_i$ with
 $d_i=0$ or $1$ for all $g_i \in \mathcal{G}(I)$, then
 Knutson, Miller, and Yong note that we get the
 geometric vertex decomposition of $I$ with respect
 to $y$ for ``free":
 
 \begin{lemma}[{\cite[Theorem 2.1 (a), (b)]{KMY}}]\label{square-freey}
 Let  $I$  be an ideal of $R$ and let $<$ be a $y$-compatible monomial order.  Suppose that 
$\mathcal{G}(I) = \{ g_1,\ldots,g_m\}$ is a Gr\"obner basis of $I$ with respect to $<$, and also suppose that 
${\rm in}_y(g_i) = y^{d_i}q_i$ with $d_i=0$ or $1$ for all
$i$. Then
\begin{enumerate}
    \item $\{q_1,\ldots,q_m \}$ is a Gr\"obner basis of $C_{y,I}$
    and $\{q_i ~|~ d_i = 0 \}$ is a Gr\"obner basis of $N_{y,I}$.
\item ${\rm in}_y(I) = C_{y,I} \cap (N_{y,I} + \langle y \rangle)$, i.e.,
$I$ has a geometric vertex decomposition with respect to $y$.
\end{enumerate}
 \end{lemma}

 \begin{remark}
 If $I$ is a square-free monomial ideal in $R$, then $I$ is geometrically
 vertex decomposable if and only if the simplicial complex $\Delta$
 associated with $I$ via the Stanley-Reisner correspondence is a
 vertex decomposable simplicial complex;  see \cite[Proposition 2.8]{KR} for more details.  As a consequence, we can view Definition \ref{gvd} as a generalization of the notion of vertex decomposability.   When $I$ is a square-free
 monomial ideal with associated
 simplicial complex $\Delta$, then $C_{y,I}$ is the Stanley-Reisner
 ideal of the star of $y$, i.e., ${\rm star}_\Delta(y)
 = \{F \in \Delta ~|~ F \cup \{y\} \in \Delta \}$ and $N_{y,I}+\langle y\rangle$ corresponds
 to the deletion of $y$ from $\Delta$, that is, 
 ${\rm del}_\Delta(y) = \{F \in \Delta ~|~ 
 y \not\in F \}$ (see \cite[Remark 2.5]{KR}).
 \end{remark}
 
 If $I$ has a geometric vertex decomposition with respect to 
 a variable $y$, we can determine some additional
 information about a reduced Gr\"obner
 basis of $I$ with respect to any $y$-compatible monomial
 order.  In the following statement, $I$
 is {\it square-free in $y$} if there is a generating
 set $\{g_1,\ldots,g_s\}$ of $I$ such that no
 term of $g_1,\ldots,g_s$ is divisible by $y^2$.
 
 \begin{lemma}[{\cite[Lemma 2.6]{KR}}]\label{square-freeiny}
 Suppose that the ideal $I$ of $R$ has a geometric vertex
 decomposition with respect to the variable $y =x_i$.
 Then $I$ is square-free in $y$.  Moreover, for any
 $y$-compatible term order, the reduced
 Gr\"obner basis of $I$ with respect to this order
 has the form $\{yq_1+r_1,\ldots,yq_k+r_k,h_1,\ldots,h_t\}$
 where $y$ does not divide any term of $q_i,r_i,h_j$
 for $i \in \{1,\ldots,k\}$ and $j\in\{1,\ldots,t\}$.
 \end{lemma}

 The following lemma and its proof helps to illustrate some
 of the above ideas.   Furthermore, since the
 definition of geometrically vertex decomposable lends
 itself to proof by induction, the following facts
 are sometimes useful for the base cases of our induction.

 \begin{lemma}  \label{simplecases}
\begin{enumerate} 
\item An an ideal $I$ of $R = \mathbb{K}[x]$ is geometrically
 vertex decomposable if and only if $I = \langle ax +b \rangle$ for
 some $a,b \in \mathbb{K}$.
 \item Let $f = c_1m_1+\cdots + c_sm_s$ be any
 polynomial in $R = \mathbb{K}[x_1,\ldots,x_n]$ with $c_i \in \mathbb{K}$ and
 $m_i$ a monomial.  If each $m_i$ is square-free, then
 $I = \langle f \rangle$ is geometrically vertex decomposable.
 In particular, if $m$ is a square-free monomial, then
 $\langle m \rangle$ is geometrically vertex decomposable.
 
\end{enumerate}
 \end{lemma}
 
 \begin{proof} (1)
 $(\Leftarrow)$  If $a =0$, or $b =0$, or both $a=b=0$, the ideal 
 $I = \langle ax+b \rangle$ satisfies Definition \ref{gvd} (1).  So, suppose
 $a,b \neq 0$.  The ideal $I$ is prime,
 so it is unmixed.  Since $x$ is the only 
 variable of $R$, and because
 there is only one monomial order on this ring, it is easy to see
 that this monomial order is $x$-compatible, and that $\{ax+b\}$ is
 a Gr\"obner basis of $I$.  So, $C_{x,I} = \langle a \rangle = \langle 1 \rangle$ and 
 $N_{x,I} = \langle 0 \rangle$.  It is straightforward to check 
 that we have
 a geometric vertex decomposition of $I$ with respect to $x$.  Furthermore,
 as ideals in $\mathbb{K}[\hat{x}] = \mathbb{K}$, $C_{x,I} = \langle 1 \rangle$ and 
 $N_{x,I} = \langle 0 \rangle$ are geometrically
 vertex decomposable by definition.  So, $I$ is geometrically vertex decomposable.
 
 $(\Rightarrow)$ Since $R = \mathbb{K}[x]$ is a principal ideal domain,
 $I = \langle f \rangle$ for some $f \in R$, i.e., 
 $f = a_dx^d + \cdots + a_1x + a_0$ with $a_i \in \mathbb{K}$. 
 Since $I$ is geometrically vertex decomposable, and because
 $x$ is the only variable of $R$, by Lemma \ref{square-freeiny},
 the ideal $I$ is square-free in $x$.  This fact then forces
 $d \leq 1$, and thus $I = \langle a_1x+a_0\rangle$ as desired.
 
 (2)  We proceed by induction on the number of variables in
 $R=\mathbb{K}[x_1,\ldots,x_n]$.  The base case $n=1$ follows
 from statement (1).    Because $I = \langle f \rangle$ is
 principal, $f$ is a Gr\"obner basis with respect
 to any monomial order.  In particular, let 
 $>$ be the lexicographic
 order on $R$ with $x_1 > \cdots > x_n$, and 
 assume $m_1 > \cdots > m_s$.  Let $y$ be the largest variable
 dividing $m_1$.  
  Then
 we can write $f$ as $f = y(c_1m_1'+\cdots + c_im'_i)
 + c_{i+1}m_{i+1}+ \cdots + c_sm_s$ for some $i$ such that
 $y$ does not divide $m_{i+1},\ldots,m_s$.  Note that $>$ is
 a $y$-compatible monomial order, and so 
 by Lemma \ref{square-freey} we 
 have ${\rm in}_y(I) = C_{y,I} \cap (N_{y,I}+\langle y \rangle)$
 with  $C_{y,I} = \langle c_1m'_1+\cdots + c_im'_i\rangle$ and
 $N_{y,I} = \langle 0 \rangle$. The ideal $N_{y,I}$ is 
 geometrically vertex decomposable in $\mathbb{K}[x_1,\ldots,\hat{y},
 \ldots,x_n]$ by definition, and $C_{y,I}$ is geometrically
 vertex decomposable in the same ring by induction.  Observe that
 $I, C_{y,I}$ and $N_{y,I}$ are also unmixed since they are principal.
 \end{proof}
 
 Theorem \ref{tensorproduct}, which is of independent interest, shows how we can treat ideals whose generators lie in different sets of variables.  We require a lemma about Gr\"obner
 bases in tensor products.  For completeness, we give
 a proof, although it follows easily from standard facts about Gr\"obner bases.

 We first need to recall a characterization of
 Gr\"obner bases using standard representations.
 Fix a monomial order $<$ on $R =\mathbb{K}[x_1,\ldots,x_n]$.
 Given $G = \{g_1,\ldots,g_s\}$ in $R$, we say $f$ {\it reduces to
 zero modulo $G$} if 
  $f$ has a {\it standard representation}
 $$f =f_1g_1+\cdots +f_sg_s ~~\mbox{with $f_i \in R$}$$
 with ${\rm multidegree}(f) \geq 
 {\rm multidegree}(f_ig_i)$
 for all $i$ with $f_ig_i \neq 0$. Here 
 $${\rm multidegree}(h) = 
 \max\{\alpha \in \mathbb{N}^n ~|~ 
 \mbox{$x^\alpha$ is a term of $h$}\},$$
 where we use the monomial order $<$ to order $\mathbb{N}^n$.  We then have the following result.
 
 \begin{theorem}[{\cite[Chapter 2.9, Theorem 3]{CLO}}]\label{gbchar}
  Let $R =\mathbb{K}[x_1,\ldots,x_n]$ with fixed
 monomial order $<$.
 A basis $G = \{g_1,\ldots,g_s\}$ of an ideal $I$ in $R$
 is a Gr\"obner basis for $I$ if and only if
 each $S$-polynomial $S(g_i,g_j)$ reduces to zero modulo
 $G$.
 \end{theorem}

 For the
 lemma below, note that if $R =\mathbb{K}[x_1,\ldots,x_n]$
 and $S = \mathbb{K}[y_1,\ldots,y_m]$, and if 
 $<$ is a monomial order on $R \otimes S :=R \otimes_{\mathbb{K}} S$, then
 $<$ induces a monomial order $<_R$ on $R$ where
 $m_1 <_R m_2$ if and only if $m_1 < m_2$, where
 we view $m_1,m_2$ as monomials of both $R$ and $R \otimes S$. 
 Here,``viewing $f\in R$ as an element of $R\otimes S$" means writing $\varphi_R(f)$ as $f$ where $\varphi_R:R\rightarrow R\otimes S$ is the natural inclusion $f\mapsto f\otimes 1$. Similarly, we let $<_S$ denote the induced
 monomial order on $S$.

 \begin{lemma}\label{grobnertensorproduct}
 Let $I \subseteq R =\mathbb{K}[x_1,\ldots,x_n]$ and 
 $J \subseteq S=\mathbb{K}[y_1,\ldots,y_m]$ be ideals.
 For any monomial order $<$ on $R \otimes S$, there exists a  
 Gr\"obner basis of $I+J$ in $R \otimes S$ which has the form
 $\mathcal{G}(I+J) = \mathcal{G}_1 \cup \mathcal{G}_2$,
 where $\mathcal{G}_1$ is a Gr\"obner basis of
 $I$ in $R$ with respect to $<_R$ but viewed as
 elements of $R \otimes S$, and $\mathcal{G}_2$
 is a Gr\"obner basis of $J$ in $S$ with 
 respect to $<_S$
 but viewed as elements of $R \otimes S$.
 \end{lemma}
 
 \begin{proof}
 Given $<$, select a Gr\"obner basis $\mathcal{G}_1$ of $I$ and $\mathcal{G}_2$ of $J$ with respect to the induced monomial orders $<_R$ and $<_S$ on $R$ and $S$ respectively. Since $\mathcal{G}_1$ generates $I$ and $\mathcal{G}_2$
 generates $J$, the set $\mathcal{G}_1 \cup \mathcal{G}_2$ generates $I+J$ as an ideal of $R \otimes
 S$.  To prove that $\mathcal{G}_1 \cup \mathcal{G}_2$
 is a Gr\"obner basis of $I+J$, by 
 Theorem \ref{gbchar} it suffices to show
 that for any $g_i,g_j \in \mathcal{G}_1 \cup \mathcal{G}_2$, the 
 $S$-polynomial $S(g_i,g_j)$ reduces
 to zero modulo this set.
 
 If $g_i,g_j \in \mathcal{G}_1$, then since $g_i,g_j \in R$, and since $\mathcal{G}_1$ is a Gr\"obner basis
 of $I$ in $R$, by Theorem \ref{gbchar}, the
 $S$-polynomial $S(g_i,g_j)$ reduces to zero modulo
 $\mathcal{G}_1$.  But then in the larger ring 
 $R \otimes S$, the $S$-polynomial
 $S(g_i,g_j)$ also reduces to zero modulo 
 $\mathcal{G}_1 \cup \mathcal{G}_2$.  A similar result
 holds if $g_i,g_j \in \mathcal{G}_2$.

So, suppose $g_i \in \mathcal{G}_1$ and $g_j \in \mathcal{G}_2$.  Note that the leading monomial of 
$g_i$ is only in the variables $\{x_1,\ldots,x_n\}$,
while the leading monomial of $g_j$ is only in
the variables $\{y_1,\ldots,y_m\}$.  Consequently,
their leading monomials are relatively prime.  Thus, by
\cite[Chapter 2.9, Proposition 4]{CLO}, the
$S$-polynomial $S(g_i,g_j)$ reduces to zero modulo
$\mathcal{G}_1 \cup \mathcal{G}_2$.
 \end{proof}
 
 \begin{theorem}\label{tensorproduct}
 Let $I \subsetneq R =\mathbb{K}[x_1,\ldots,x_n]$ and 
 $J \subsetneq S=\mathbb{K}[y_1,\ldots,y_m]$ be proper ideals. Then $I$ and $J$ are geometrically
 vertex decomposable if and only if $(I+J)$ is geometrically vertex decomposable in $R \otimes S =\mathbb{K}[x_1,\ldots,x_n,y_1,\ldots,y_m]$.
 \end{theorem}
 
 \begin{proof}
 First suppose that $I\subsetneq R$ and $J\subsetneq S$ are geometrically vertex decomposable. Since neither ideal
 contains $1$, we have $I+J \neq \langle 1 \rangle$.   By \cite[Corollary 2.8]{HNTT},
 the set of associated primes of $(R \otimes S)/(I+J)
 \cong R/I \otimes S/J$ satisfies
 \begin{equation}\label{assprimes}
 {\rm Ass}_{R \otimes S}(R/I \otimes S/J)
 = \{P+Q ~|~ P \in {\rm Ass}_R(R/I) ~~
 \mbox{and} ~ Q \in {\rm Ass}_S(S/J)\}.
 \end{equation}
 Thus any associated prime
 $P+Q$ of $(R\otimes S)/(I+J)$ satisfies
\begin{eqnarray*}\dim((R\otimes S)/(P+Q)) &=
&\dim(R/P)+ \dim(S/Q) \\
&=& \dim(R/I)+\dim(S/J) \\
&=&
\dim((R \otimes S)/(I+J))
\end{eqnarray*}
where we are using the fact that $I$ and $J$
are unmixed for the second equality.  So,
$I+J$ is also unmixed.
 
 To see that $I+J\subseteq R\otimes S$ is geometrically vertex
 decomposable, we proceed by induction on the number of variables $\ell = n+m$ in $R\otimes S$. The base case $\ell= 0$ is trivial.
 Assume now that $\ell>0$. 
 If both $I$ and $J$ are generated by indeterminates, then $I+J$ is too and so is geometrically vertex decomposable. Thus, without loss of generality, suppose that $I$ is not generated by indeterminates (note that $I \neq \langle 1 \rangle$ by assumption).
 
 Because $I$ is geometrically vertex decomposable in $R$,
 there is a variable $y = x_i$ in $R$ such that $\text{in}_y (I) = C_{y,I}\cap (N_{y,I}+\langle y\rangle)$ is a geometric vertex decomposition and the contractions of $C_{y,I}$ and $N_{y,I}$ to $R' = \mathbb{K}[x_1,\ldots, \hat{y},\ldots, x_n]$ are geometrically vertex decomposable.   Extend the $y$-compatible
 monomial order $<$ on $R$ to a $y$-compatible
 monomial order on $R \otimes S$ by taking any monomial
 order  on $S$, and let our new monomial
 order $\prec$  be the product order of these two monomial orders (where
 $x_i \succ y_j$ for all $i,j$).

If we write $K^e$ to denote the extension of an ideal
$K$ in $R$ into the ring $R \otimes S$, then
one checks that with respect to this new
$y$-compatible order
 \begin{eqnarray*}
 \text{in}_y(I+J) &=& (\text{in}_y (I))^e+ J = [C_{y,I}\cap (N_{y,I}+\langle y\rangle)]^e+J \\
 &=& ((C_{y,I})^e+J)\cap ((N_{y,I})^e+J + \langle y\rangle).
 \end{eqnarray*}
 Using the identities
 $$(C_{y,I})^e+J = C_{y,I+J} ~~\mbox{and}~~ (N_{y,I})^e +J
 = N_{y,I+J}$$
(note that $\prec$ is being used to define $C_{y,I+J}$ and
$N_{y,I+J}$ and $<$ is being used to define $C_{y,I}$ and
$N_{y,I}$),
 we have a geometric vertex decomposition of $I+J$
 with respect to $y$ in $R \otimes S$:
$$ \text{in}_y(I+J) =  C_{y,I+J} \cap (N_{y,I+J} + \langle y \rangle).$$
 
Now let $C'$ and $N'$ denote the contractions of $C_{y,I}$ and $N_{y,I}$ to $R'$. First assume that $C'$ and $N'$ are both proper ideals. Then, since $C'$ and $N'$ are geometrically vertex decomposable, we may apply induction to see that $C'+J$ and $N'+J$ in $R'\otimes S$ are geometrically vertex decomposable. 
In particular, as $C'+J$ and $N'+J$ are the contractions of $(C_{y,I})^e+J$ and $(N_{y,I})^e+J$ to $R'\otimes S$, we have that $I+J$ is geometrically vertex decomposable by induction. If either $C'$ or $N'$ is
the ideal $\langle 1 \rangle$, the same would be true
for the contractions of $(C_{y,I})^e+J$ or $(N_{y,I})^e+J$ because
the contraction of $(C_{y,I})^e+J$, respectively $(N_{y,I})^e+J$,
contains $C'$, respectively $N'$.  So $I+J$ is geometrically vertex
decomposable.

 For the converse, we proceed by induction on the number of variables $\ell$ in $R\otimes S$. The base case is $\ell=0$, which is trivial.
 So suppose $\ell>0$.  We first show that $I$
 is unmixed. Suppose that $I$ is not unmixed; that is, there are associated primes 
 $P_1$ and $P_2$ of ${\rm Ass}(R/I)$
 such that $\dim(R/P_1) \neq \dim(R/P_2)$.  For any
 associated prime $Q$ of $S/J$, we know
 by \eqref{assprimes} that $P_1+Q$
 and $P_2+Q$ are associated primes of
 $(R \otimes S)/(I+J)$.  Since
 $I+J$ is unmixed, we can
 derive the contradiction
\begin{eqnarray*}
\dim((R\otimes S)/(I+J))& =& \dim((R \otimes S)/(P_1+Q)) \\
&=& \dim (R/P_1) + \dim (S/Q) \\ 
&\neq& \dim(R/P_2) + \dim(S/Q) \\
&=& \dim((R\otimes S)/(P_2+Q)) = \dim((R\otimes S)/(I+J)).
\end{eqnarray*}
 So, $I$ is unmixed (the proof for $J$ is similar).
 
 If $I+J$ is generated by indeterminates, then so are $I$ and $J$, hence they are geometrically vertex decomposable. 
 So, suppose that there is a variable $y$ in
$R \otimes S$ and a $y$-compatible monomial order
$<$ such that 
$${\rm in}_y(I+J) = C_{y,I+J} \cap (N_{y,I+J} + \langle y \rangle).$$
Without loss of generality, assume that
$y \in \{x_1,\ldots,x_n\}$.   So
$C_{y,I+J}$ and $N_{y,I+J}$ are geometrically
vertex decomposable in $\mathbb{K}[x_1,\ldots,\hat{y},\ldots,x_n,
y_1,\ldots,y_m]$.

 By Lemma \ref{grobnertensorproduct}, we can construct a Gr\"obner basis $\mathcal{G}$ of $I+J$ with
respect to $<$ such that 
$$\mathcal{G} = \{g_1,\ldots,g_s\} \cup \{h_1,\ldots,h_t\}$$
where $\{g_1,\ldots,g_s\}$ is a Gr\"obner basis 
of $I$ with respect to the order $<_R$ in $R$, and $\{h_1,\ldots,h_t\}$ is a Gr\"obner basis of
$J$ with respect to $<_S$ in $S$.  Since $y$ can only appear among the $g_i$'s,
we have 
$$C_{y,I+J} = (C_{y,I}) +J~~\mbox{and}~~ N_{y,I+J} = (N_{y,I})+J$$
where $C_{y,I}$, respectively $N_{y,I}$, denote the
ideals constructed from the Gr\"obner basis $\{g_1,\ldots,g_s\}$ of $I$ in $R$ using the monomial
order $<_R$.  Note that in $R$, $<_R$ is still $y$-compatible.

Since the ideals
$(C_{y,I})+J$ and $(N_{y,I})+J$ are geometrically
vertex decomposable in the ring  $\mathbb{K}[x_1,\ldots,\hat{y},\ldots,x_n,
y_1,\ldots,y_m]$, by induction, $C_{y,I}$ and $N_{y,I}$ 
are geometrically vertex decomposable in 
$\mathbb{K}[x_1,\ldots,\hat{y},\ldots,x_n]$ and $J$ is geometrically vertex decomposable in $S$.
To complete the proof, note that in $R$, we 
have ${\rm in}_y(I) = C_{y,I} \cap (N_{y,I}+\langle y \rangle)$.  Thus $I$ is also geometrically
vertex decomposable in $R$.
 \end{proof}
 
 \begin{remark}
 If we weaken the hypotheses in Theorem \ref{tensorproduct}
 to allow $I$ or $J$ to be $\langle 1 \rangle$, then only
 one direction remains true.  In particular, if $I$ and $J$
 are geometrically vertex decomposable, then so is $I+J$.  However, the
 converse statement would no longer be true.  To see why,
 let $I = \langle 1 \rangle$ and let $J$ to be any ideal which is not  geometrically vertex decomposable.  Then $I+J = \langle
 1 \rangle$ is geometrically vertex decomposable
in $R \otimes S$, but we do not have that both $I$ and $J$ are 
geometrically vertex decomposable.
 \end{remark}
 
 \begin{remark}\label{join}
Theorem \ref{tensorproduct} is an algebraic
generalization
of \cite[Proposition 2.4]{PB} which showed that if $\Delta_1$ and $\Delta_2$ were simplicial complexes on different sets of variables,
then the join $\Delta_1 \star \Delta_2$ is vertex
decomposable if and only if $\Delta_1$ and $\Delta_2$
are vertex decomposable.
 \end{remark}
 
\begin{corollary}\label{monomialcor}
Let $I \subseteq R =\mathbb{K}[x_1,\ldots,x_n]$ be a square-free monomial ideal.
If $I$ is a complete intersection, then $I$ is geometrically vertex decomposable.
\end{corollary} 
 
 \begin{proof}
 Suppose $I = \langle m_1,\ldots,m_t\rangle$, where $m_1,\ldots,m_t$ are the minimal square-free monomial generators. Because $I$ is a complete
 intersection, the ideal is unmixed.  Furthermore,
 because $I$ is a complete
 intersection, the support of each monomial is pairwise disjoint.  So,
 after a relabelling, we can assume, $m_1 = x_1x_2\cdots x_{a_1}$,
 $m_2 = x_{a_1+1}\cdots x_{a_2}, \ldots, m_t = x_{a_{t-1}+1}\cdots x_{a_t}$.
 Then
 $$R/I \cong \mathbb{K}[x_1,\ldots,x_{a_1}]/\langle m_1 \rangle \otimes \cdots \otimes
\mathbb{K}[x_{a_{t-1}+1},\ldots,x_{a_t}]/\langle m_t \rangle \otimes \mathbb{K}[x_{a_{t+1}},\ldots,x_n].$$
By Lemma \ref{simplecases}, the ideals $\langle m_i \rangle$ 
are geometrically vertex decomposable for $i=1,\ldots,t$.
Now repeatedly apply Theorem \ref{tensorproduct}.
 \end{proof}
 
 \begin{remark}
 Corollary \ref{monomialcor} can also be deduced via results from Stanley-Reisner
 theory, which we sketch out.  
 One proceeds by induction on the number of generators of the complete intersection
 $I$.  If $I = \langle x_{1}\cdots x_k \rangle$ 
 has one generator, then one can prove directly from the definition of
 a vertex decomposable simplicial complex (e.g. see \cite{PB}), that 
 the simplicial complex associated with $I$, denoted by $\Delta  = \Delta(I)$, is vertex 
 decomposable.  For the induction step, note that if $I = \langle m_1,\ldots, m_t
 \rangle$, then  $I = I_1 + I_2 = \langle m_1,\ldots,m_{t-1}\rangle 
 + \langle m_t\rangle$.  If $\{w_1,\ldots,w_m\}$ are variables that
 appear in the generator $m_t$ and $\{x_1,\ldots,x_\ell\}$ are the other variables,
 then we have 
 $$R/I \cong \mathbb{K}[x_1,\ldots,x_\ell]/I_1 \otimes \mathbb{K}[w_1,\ldots, w_m]/I_2.$$
 By induction, the simplicial complexes $\Delta_1$ and $\Delta_2$ defined by $I_1$ and $I_2$ are vertex decomposable.  As noted in Remark \ref{join}, the join $\Delta_1 \star \Delta_2$ is also vertex decomposable.  So, the ideal $I$ is a square-free monomial ideal whose associated simplicial complex is vertex decomposable.  The result now
 follows from 
 \cite[Theorem 4.4]{KR} which implies that the ideal $I$ is also
 geometrically vertex decomposable.
 \end{remark}

\section{Toric ideals of graphs}

This section initiates a study of the geometric
vertex decomposability of toric ideals of graphs.  We have
subdivided this section into three parts:  (a) a review
of the needed background on toric ideals, (b) an 
analysis of the ideals $C_{y,I}$ and $N_{y,I}$ when
$I$ is the toric ideal of a graph, 
and (c) an explanation of
how the graph operation of ``gluing'' a cycle to a graph preserves geometric vertex decomposability.

We will study some specific families of graphs whose toric ideals are geometrically vertex decomposable in Sections \ref{sec:bipartite} and  \ref{section_square-free}.

\subsection{Toric ideals of graphs} We review the relevant
background on toric ideals of graphs.
Our main references for this material are \cite{Sturm,V}.

Let $G = (V(G),E(G))$ be a finite simple graph with 
vertex set $V(G) =\{x_1,\ldots,x_n\}$ and edge
set $E(G) = \{e_1,\ldots,e_t\}$ where each $e_i = \{x_j,x_k\}$.
Let $\mathbb{K}[E(G)] = \mathbb{K}[e_1,\ldots,e_t]$ be a polynomial
ring, where we treat the $e_i$'s as indeterminates.  Similarly,
let $\mathbb{K}[V(G)] = \mathbb{K}[x_1,\ldots,x_n]$.  Consider the 
$\mathbb{K}$-algebra homomorphism $\varphi_G:\mathbb{K}[E(G)] \rightarrow \mathbb{K}[V(G)]$ given
by
$$\varphi_G(e_i) = x_jx_k  ~~\mbox{where $e_i = \{x_j,x_k\}$ for
all $i \in \{1,\ldots,t\}$}.$$
The {\it toric ideal of the graph $G$}, denoted $I_G$,
is the kernel of the homomorphism $\varphi_G$.

While the generators of $I_G$ are defined implicitly, 
these generators (and a Gr\"obner basis) of $I_G$ can be
described in terms  of the graph $G$, specifically,
the walks in $G$.  A {\it walk} of length $\ell$ is an alternating sequence of 
vertices and edges $$\{x_{i_0},e_{i_1},x_{i_1},e_{i_2},\cdots,e_{i_\ell},
x_{i_{\ell}}\}$$ 
such that $e_{i_j} = \{x_{i_{j-1}},x_{i_j}\}$. The walk
is {\it closed} if $x_{i_\ell} = x_{i_0}$.
When the vertices are clear, we simply write the walk as $\{e_{i_1},\ldots,e_{i_\ell}\}$.  It straightforward
to check that every closed walk of even length,
say $\{e_{i_1},\ldots,e_{i_{2\ell}}\}$, 
results in an element of $I_G$;  indeed
$$\varphi_G(e_{i_1}e_{i_3}\cdots e_{i_{2\ell-1}} - 
e_{i_2}e_{i_4}\cdots e_{2\ell}) =
x_{i_0}x_{i_1}\cdots x_{2\ell-1} - x_{i_1}x_{i_2}\cdots x_{i_{2\ell}} =0
$$
since $x_{i_{2\ell}}=x_{i_0}$.  Note that
$e_{i_1}e_{i_3}\cdots e_{i_{2\ell-1}} - 
e_{i_2}e_{i_4}\cdots e_{i_{2\ell}}$ is a binomial.
 For any $\alpha = (a_1,\ldots,a_t) \in \mathbb{N}^t$, let
$e^\alpha = e_1^{a_1}e_2^{a_2}\cdots e_t^{a_t}$.  A binomial
$e^\alpha-e^\beta \in I_G$ is {\it primitive} if there
is no other binomial $e^\gamma-e^\delta \in I_G$ such
that $e^\gamma|e^\alpha$ and $e^\delta|e^\beta$.
We can now describe generators and a universal Gr\"obner
basis of $I_G$.

\begin{theorem}\label{generatordescription}
Let $G$ be a finite simple graph.  
\begin{enumerate}
    \item {\cite[Proposition 10.1.5]{V}}  The ideal
    $I_G$ is generated by the set of binomials 
$$ \{e_{i_1}e_{i_3}\cdots e_{i_{2\ell-1}} - 
e_{i_2}e_{i_4}\cdots e_{i_{2\ell}} ~~|~~ \mbox{$\{e_{i_1},\ldots,e_{i_{2\ell}}\}$ is a closed even walk of $G$}\}.$$
\item {\cite[Proposition 10.1.9]{V}}  The set of all primitive
binomials that also correspond to closed even walks in $G$
is a universal Gr\"obner basis of $I_G$.
\end{enumerate}
\end{theorem}

\noindent Going forward, we will write $\mathcal{U}(I_G)$ to denote
this universal Gr\"obner basis of $I_G$.

The next two results allow us to make some additional
assumptions on $G$ when studying $I_G$.
First, we can ignore leaves in $G$
when studying $I_G$.
Recall that the degree of a vertex $x \in V(G)$ is the number of edges $e \in E(G)$ that contain $x$.   An edge $e= \{x,y\}$ is a {\it leaf} of $G$  if either $x$ or $y$ has degree one.   In the statement below, if $e \in E(G)$, then
by $G\setminus e$ we mean the graph formed by removing the 
edge $e$ from $G$; note $V(G \setminus e) = V(G)$. 
We include a proof for completeness.

\begin{lemma}\label{removeleaves}
Let $G$ be a finite simple graph.  If $e$ is a leaf of $G$, then $I_{G} = I_{G \setminus e}$.
\end{lemma}

\begin{proof}
For the containment $I_{G\setminus e} \subseteq I_G$, 
observe that any closed even walk in $G\setminus e$ is also a closed
even walk in $G$.
For the reverse containment, if a closed even
walk $\{e_{i_1}, \ldots,e,\ldots,e_{i_{2\ell}}\}$ contains the leaf $e$, then $e$ must be repeated, i.e.,  $\{e_{i_1}, \ldots,e,e,\ldots,e_{i_{2\ell}}\}$.  The corresponding
binomial $b_1-b_2$ is divisible by $e$, 
i.e., $b_1-b_2 = e(b_1'-b_2') \in I_G$. But since $I_G$ is a prime binomial ideal, this forces $b'_1-b'_2 \in I_G$. Thus every minimal generator of $I_G$ corresponds
to a closed even walk that does not go through $e$,
and thus is an element of $I_{G\setminus e}$.
\end{proof}

A graph $G$ is {\it connected} if for any two pairs
of vertices in $G$, there is a walk in $G$ between these two
vertices.  A connected component of $G$ is a subgraph of $G$ that is
connected, but it is not contained in any larger connected subgraph.
To study the geometric vertex decomposability of $I_G$, we may always assume
that $G$ is connected.

\begin{theorem}\label{connected}
Suppose that $G = H \sqcup K$ is the disjoint union of 
two finite simple graphs.  Then $I_G$ is geometrically vertex decomposable
in $\mathbb{K}[E(G)]$
if and only if $I_H$, and respectively $I_K$, is geometrically
vertex decomposable in $\mathbb{K}[E(H)]$, and respectively $\mathbb{K}[E(K)]$.
\end{theorem}

\begin{proof}
Apply Theorem \ref{tensorproduct}
to $I_G = I_H+I_K$ in $\mathbb{K}[E(G)] = \mathbb{K}[E(H)] \otimes \mathbb{K}[E(K)]$.
\end{proof}
 
The  well-known result below gives a condition
for $\mathbb{K}[E(G)]/I_G$ to be  Cohen-Macaulay.

\begin{theorem}\label{sqfree=>cm}
Let $G$ be a finite simple graph with toric ideal $I_G \subseteq \mathbb{K}[E(G)]$.  Suppose
that there is a monomial order $<$ such that
${\rm in}_<(I_G)$ is a square-free monomial ideal.
Then $\mathbb{K}[E(G)]/I_G$ is Cohen-Macaulay.
\end{theorem}

\begin{proof}
If $\init_{<}(I_G)$ is a square-free monomial ideal, 
then $I_G$ is normal by \cite[Proposition 13.15]{Sturm}. Thus, by Hochster \cite{H}, $\mathbb{K}[E(G)]/I_G$ is  Cohen-Macaulay. 
\end{proof}

\subsection{Structure results about $N_{y,I}$ and $C_{y,I}$}
To study the geometric vertex decomposability of $I_G$,
we need access to both $N_{y,I_G}$ and $C_{y,I_G}$.  While
determining $C_{y,I_G}$ in terms of $G$ will prove to be subtle, the ideal $N_{y,I_G}$ has a straightforward
description. 

\begin{lemma}\label{linktoricidealgraph}
Let $G$ be a finite simple graph with toric ideal $I_G \subseteq \mathbb{K}[E(G)]$.
Let $<$ by any $y$-compatible monomial order with $y = e$ for some
edge $e$ of $G$.  Then
$$N_{y,I_G} = I_{G\setminus e}.$$
In particular, a universal Gr\"obner basis 
of $N_{y,I_G}$ consists of all the binomials in the universal
Gr\"obner basis $\mathcal{U}(I_G)$ of $I_G$ where neither term is divisible by $y$.
\end{lemma}

\begin{proof}
By Theorem \ref{generatordescription} (2), $I_G$ has a universal Gr\"obner basis $\mathcal{U}(I_G)$
of primitive binomials
associated to closed even walks of $G$. 
Write this basis as $\mathcal{U}(I_G)  = \{ y^{d_1}q_1 +r_1,\ldots, y^{d_k}q_k +r_k, g_1,\ldots, g_r \}$, where $d_i>0$ and where $y$ does not divide any term of $g_i$ and $q_i$. By definition \[N_{y,I_G} = \langle  g_1,\ldots, g_r \rangle.\]

In particular, $N_{y,I_G}$ is generated by primitive binomials in $\mathcal{U}(I_G)$ which do not include the variable $y$.  These
primitive binomials correspond to closed even walks in $G$ which do not pass through the edge $e$. In particular, they are also closed even walks in $G\setminus e$, so $ \{ g_1,\ldots, g_r\}\subset  \mathcal{U}(I_{G\setminus e})$, the universal Gr\"obner basis of $I_{G\setminus e}$ from Theorem \ref{generatordescription} (2).

To show the reverse containment $\mathcal{U}(I_{G\setminus e})\subseteq \{ g_1,\ldots, g_r\}$, suppose that there is some binomial $u-v\in \mathcal{U}(I_{G\setminus e})$ which is not in $\mathcal{U}(I_G)$.  Then there would be some closed even walk of $G$ which is not primitive, but becomes primitive after deleting the edge $e$. For $u-v$ to not be primitive means that there is some primitive 
binomial $u'-v' \in \mathcal{U}(I_G)$ such that $u' | u$ and $v'|v$. Since $y$ does not divide $u$ or $v$, we must have $u'-v' \in  \mathcal{U}(I_{G\setminus e})$, a contradiction to $u-v$ being primitive. Therefore $\mathcal{U}(I_{G\setminus e})=\{ g_1,\ldots, g_r\}$.  Since $\{g_1,\ldots,g_r\}$ generates $I_{G\setminus e}$, we have $I_{G\setminus e}= \langle g_1,\ldots,g_r \rangle
= N_{y,I_G} $, thus proving the result.
\end{proof}

It is more difficult to give a similar description for $C_{y,I_G}$. For example, $C_{y,I_G}$ may not be prime, and thus, it may not be the toric ideal of any graph.  
If we make the extra assumption that the binomial generators in $\mathcal{U}(I_G)$ are 
{\it doubly square-free} (i.e., each binomial is the difference of two square-free monomials), then it is possible to give a slightly more concrete description of $C_{y,I_G}$.  We work out these details below.

Fix a variable $y$ in $\mathbb{K}[E(G)]$, and write the elements
of $\mathcal{U}(I_G)$ as  $\{y^{d_1}q_1 +r_1,\ldots, y^{d_k}q_k +r_k, g_1,\ldots, g_r \}$, where $d_i>0$ and where $y$ does not divide $q_i$ or any term of $g_i$.  Since we are assuming
the elements in $\mathcal{U}(I_G)$ are doubly square-free, we have $d_i = 1$
for $i=1,\ldots,k$ and $q_1,\ldots,q_k$ are square-free monomials.
Consequently
 \[{\rm in}_y(I_G)=\langle yq_1,\ldots,yq_k, g_1,\ldots,g_r\rangle\]
 is generated by doubly square-free binomials and square-free monomials.
 Let $\bigcap_j Q_j$ be the primary decomposition of $\langle yq_1,
 \ldots, yq_k \rangle$.  Each $Q_j$ is an ideal generated by
 variables since $\langle yq_1,\ldots,yq_k \rangle$ is a
 square-free monomial ideal.  Thus
 $${\rm in}_y(I_G) = \left( \bigcap_j Q_j \right) + \langle g_1,\ldots,g_r
 \rangle = \bigcap_j (Q_j+ \langle g_1,\ldots,g_r \rangle).$$
 If there is a $g_l=u_l-v_l$ with either $u_l$ or $v_l\in Q_j$, then $Q_j +\langle g_1,\ldots,g_r\rangle$ can be further decomposed
into an intersection of ideals generated by variables and square-free binomials.

 Continuing this process, we can 
 write ${\rm in}_y(I_G)=\bigcap_i P_i$, where
 each $P_i = M_i+T_i$, with $M_i$ an ideal generated by a subset of indeterminates in $\{e_1,\ldots,e_t\}$, and $T_i\subseteq\mathcal{U}(I_G)$ is an ideal of binomials generated by $g_l = u_l-v_l$ where $u_l,v_l\notin M_i$.   Again, we point
 out that each binomial is a doubly square-free binomial
 by our assumption on $\mathcal{U}(I_G)$. As the next result shows, the binomial
 ideal $T_i$ is a toric ideal corresponding to
 a subgraph of $G$.

\begin{theorem}\label{link_components}
Let $G$ be a finite simple graph with toric ideal $I_G \subseteq \mathbb{K}[E(G)]$, and suppose that the elements of 
$\mathcal{U}(I_G)$ are doubly square-free.  For a fixed
variable $y$ in $\mathbb{K}[E(G)]$, suppose that 
$${\rm in}_y(I_G) = \bigcap_i P_i ~~\mbox{with
$P_i = M_i +T_i$},$$
using the notation as above.  Let $E_i \subseteq E(G)$
be the set of edges that correspond to the variables in $M_i + \langle y \rangle$,
and let $G \setminus E_i$ be the graph $G$ with all the edges of
$E_i$ removed.
Then $T_i = I_{G\setminus E_i}$.
\end{theorem}

\begin{proof}
The generators of $T_i$ are those elements of $\mathcal{U}(I_G)$ whose terms  are not divisible by any variable contained in $M_i + \langle y \rangle$. So a generator 
of $T_i$ corresponds
to a primitive closed even walk that does not contain any
of the edges in $E_i$. Therefore, each generator of $T_i$ is a closed
even walk in $G \setminus E_i$, and thus
 $T_i\subset I_{G\setminus E_i}$ by 
Theorem \ref{generatordescription} (1). Conversely, suppose that $\Gamma \in \mathcal{U}(I_{G\setminus E_i})$. Then by Theorem  \ref{generatordescription} (2), $\Gamma$ corresponds to some primitive closed even walk of $G$ not passing through any edge of $E_i$. These are exactly the generators in $T_i$.
\end{proof}

We now arrive at a primary decomposition of ${\rm in}_y(I_G)$.

\begin{corollary}\label{decomposition1}
Let $G$ be a finite simple graph with toric ideal $I_G \subseteq \mathbb{K}[E(G)]$, and suppose that the elements of 
$\mathcal{U}(I_G)$ are doubly square-free.
For a fixed
variable $y$ in $\mathbb{K}[E(G)]$, suppose that 
$${\rm in}_y(I_G) = \bigcap_i P_i ~~\mbox,$$
using the notation as above.
Then each $P_i$ is a prime ideal, and after removing redundant components, this intersection defines a primary
decomposition of ${\rm in}_y(I_G)$.
\end{corollary}
\begin{proof}
By the previous result,  $P_i = M_i + I_{G\setminus E_i}$ for every $i$.  
So the fact that $P_i$ is a prime ideal 
immediately follows from the fact that any toric ideal is prime, and that no cancellation occurs between variables in $M_i$ and elements of 
$T_i = I_{G\setminus E_i}$.
\end{proof}

If $I_G$ is generated by a doubly square-free universal Gr\"obner basis, choosing any $y=e_i$ defines a geometric vertex decomposition of $I_G$ with respect to $y$ by Lemma~\ref{square-freey}. Note that $\langle y \rangle$ appears in the primary decomposition of
$\langle yq_1,\ldots,yq_k\rangle$, 
so one prime ideal in the decomposition given in
Corollary \ref{decomposition1}
${\rm in}_y(I_G)$ will always be $\langle y \rangle + \langle g_1,\ldots,g_r\rangle$.  But this is exactly $\langle y\rangle +N_{y,I_G} = \langle y \rangle + I_{G\setminus e}$, by Theorem \ref{linktoricidealgraph}.  As the next theorem
shows, if we omit this prime ideal, the remaining prime ideals form
a primary decomposition of $C_{y,I_G}$.

\begin{theorem}\label{toric_GVD}
Let $G$ be a finite simple graph with toric ideal $I_G \subseteq \mathbb{K}[E(G)]$, and suppose that the elements of 
$\mathcal{U}(I_G)$ are doubly square-free.
Fix any  variable $y=e_i$.  Suppose that
after relabelling
the primary decomposition 
${\rm in}_y(I_G)$ of Corollary \ref{decomposition1} 
we have 
\begin{equation}\label{square-freedecomp}{\rm in}_y(I_G)=\bigcap_{i=0}^d(M_i+I_{G\setminus E_i}) = (\langle y \rangle + I_{G\setminus e_i}) \cap
\bigcap_{i=1}^d 
(M_i+I_{G\setminus E_i}).
\end{equation}
Then $$C_{y,I_G} = \bigcap_{i=1}^d(M_i+I_{G\setminus E_i})$$ is a primary decomposition for $C_{y,I_G}$. Furthermore, if $<$ is a $y$-compatible monomial order,
then \eqref{square-freedecomp} is a geometric vertex decomposition for $I_G$ with respect to $y$. 
\end{theorem}
\begin{proof}
The fact about the geometric vertex decomposition follows from Lemma \ref{square-freey}.  

Since $\mathcal{U}(I_G)$ contains doubly square-free binomials, we can write 
\begin{eqnarray*}
{\rm in}_y(I_G) &=& \langle ym_1,\ldots,ym_k, g_1,\ldots,g_r\rangle 
 = \langle y,g_1\ldots,g_r \rangle \cap \langle m_1,\ldots,m_k,g_1,\ldots,g_r \rangle
\end{eqnarray*}where $y$ does not divide any $m_i$ or any term of any $g_i$. By definition, 
\[N_{y,I_G} = \langle g_1,\ldots, g_r \rangle \hspace{4mm} 
~\mbox{and}~C_{y,I_G} =  \langle  m_1,\ldots, m_k, g_1,\ldots, g_r \rangle.\] 
Applying the process described before Theorem \ref{link_components} to $\langle  m_1,\ldots, m_k, g_1,\ldots, g_r \rangle$ proves the first claim. 
\end{proof}

\begin{remark}\label{structure_remark}
Let $M$ be a square-free monomial ideal and $I_H$ a toric ideal of a graph $H$ where elements of $\mathcal{U}(H)$ are doubly square-free. The arguments presented above can be adapted to prove that $M+I_H$ has a primary decomposition into prime ideals of the form $M_i+T_i$ as in Theorem~\ref{link_components}. 
\end{remark}

\subsection{Geometric vertex decomposability under graph operations}
Given a graph $G$ whose toric ideal $I_G$ is 
geometrically vertex decomposable, it is natural
to ask if there are any graph operations we can
perform on $G$ to make a new graph $H$ so that
the associated toric ideal $I_H$ is also
geometrically vertex decomposable.   We show that the operation
of ``gluing'' an  even cycle onto a graph
$G$ is one such operation.

We make this more precise.  Given a graph $G = (V(G),E(G))$ and
a subset $W \subseteq V(G)$, the {\it induced graph} 
of $G$ on $W$,
denoted $G_W$, is the graph $G_W = (W,E(G_W))$ where 
$E(G_W) = \{ e\in E(G) ~|~ e \subseteq W\}$.   A graph $G$ is
a {\it cycle} (of length $n$) if $V(G) = \{x_1,\ldots,x_n\}$ and
$E(G) = \{\{x_1,x_2\},\{x_2,x_3\},\ldots,\{x_{n-1},x_n\},\{x_n,x_1\}\}.$

Following \cite[Construction 4.1]{FHKVT}, we define the {\it gluing}
of two graphs as follows.  Let $G_1$ and $G_2$ be two graphs,
and suppose that $H_1 \subseteq G_1$ and $H_2 \subseteq G_2$ are
induced subgraphs of $G_1$ and $G_2$ that are isomorphic.  If
$\varphi:H_1 \rightarrow H_2$ is the corresponding graph isomorphism,
we let $G_1 \cup_\varphi G_2$ denote the disjoint union $G_1 \sqcup G_2$
with the associated edges and vertices of $H_1 \cong H_2$ being identified.
We may say $G_1$ and $G_2$ are {\it glued along $H$} if both the induced
subgraphs $H_1 \cong H_2 \cong H$ and $\varphi$ are clear.

\begin{example} Figure \ref{fig:onecycle} (which is adapted from \cite{FHKVT})
shows
the gluing of a cycle $C$ of even length onto
a graph $G$ to make a new graph $H$.  The labelling is included to help
illuminate the proof of the next theorem. In this figure, the cycle $C$ has
edges $f_1,f_2,\ldots,f_{2n}$.  The edge $e$ is part of the graph $G$.  We have
glued $C$ and $G$ along the edge $e \cong f_{2n}$. 
 \begin{figure}[!ht]
    \centering
    \begin{tikzpicture}[scale=0.45]
      \draw[dotted] (6,0) circle (3cm) node{$G$};
      \draw[thin,dashed] (4,1) -- (5,2);
      \draw[thin,dashed] (4,-1) -- (5,-2);
      \draw (4,-1) -- (3.3,-2.25) node[midway,left] {\footnotesize{$f_{2n-1}$}};
      \draw (4,1) -- (4,-1) node[midway, right] {$e$};
      \draw (4,1) -- (4,-1) node[midway,left] {\footnotesize{$f_{2n}$}};
      \draw (4,1) -- (3.3,2.25) node[midway,left] {\footnotesize{$f_{1}$}};

      \node at (0,0) {$C$};
      \draw[loosely dotted] (3.5,2) edge[bend right=10] (2,3.5);
      \draw[dashed] (1,4) -- (2,3.5);
      \draw (-1,4) -- (1,4);
      \draw[dashed] (-1,4) -- (-2,3.5);
      \draw[loosely dotted] (-2,3.5) edge[bend right=10] (-3.5,2);
      \draw[dashed] (-4,1) -- (-3.5,2);
      \draw (-4,1) -- (-4,-1) node[midway,left] {\footnotesize{$f_i$}};
      \draw[dashed](-4,-1) -- (-3.5,-2);
      \draw[loosely dotted] (-2,-3.5) edge[bend right=-10] (-3.5,-2);
      \draw[dashed] (-1,-4) -- (-2,-3.5);
      \draw (-1,-4) -- (1,-4);
      \draw[dashed] (1,-4) -- (2,-3.5);
      \draw[loosely dotted] (3.5,-2) edge[bend right=-10] (2,-3.5);
    
      \fill[fill=white,draw=black] (-4,1) circle (.1);
      \fill[fill=white,draw=black] (-4,-1) circle (.1);
      \fill[fill=white,draw=black] (4,1) circle (.1);
      \fill[fill=white,draw=black] (4,-1) circle (.1);
      \fill[fill=white,draw=black] (-1,4) circle (.1);
      \fill[fill=white,draw=black] (1,4) circle (.1);
      \fill[fill=white,draw=black] (-1,-4) circle(.1);
      \fill[fill=white,draw=black] (1,-4) circle (.1);
      \fill[fill=white,draw=black] (3.3,2.25) circle (.1);
      \fill[fill=white,draw=black] (3.3,-2.25) circle (.1);
    \end{tikzpicture}
    \caption{Gluing an even cycle  $C$ to a graph $G$ along an edge.}
    \label{fig:onecycle}
 \end{figure}
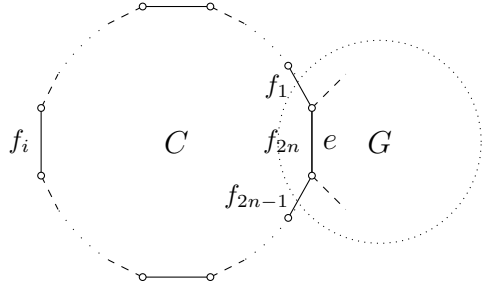
 \end{example}
 
 The geometric vertex decomposability property
 is preserved when an even cycle is glued along an edge of a graph whose toric ideal
 is geometrically vertex decomposable.

\begin{theorem}\label{gluetheorem}
Suppose that $G$ is a graph such that $I_G$ is
geometrically vertex decomposable
in $\mathbb{K}[E(G)]$.  Let $H$ be the 
graph obtained from $G$ by gluing a cycle of even length
onto an edge of $G$ (as in Figure \ref{fig:onecycle}).  Then $I_H$ is
geometrically vertex decomposable in $\mathbb{K}[E(H)]$.
\end{theorem}

\begin{proof}
The ideal $I_H$ is clearly unmixed since $I_H$ is a prime ideal.  Now
let $E(G) = \{e_1,\ldots,e_s\}$ denote the edges of $G$ and 
let $E(C) = \{f_1,\ldots,f_{2n}\}$ denote the edges of the
even cycle $C$.  Let $e$ be any edge of $G$, and after relabelling
the $f_i$'s
we can assume that $C$ is glued to $G$ along $f_{2n}$ and $e$ (see
Figure \ref{fig:onecycle}). Consequently,
$$E(H) = E(G) \cup \{f_1,\ldots,f_{2n-1}\}.$$

Let $e = f_{2n} = \{a,b\}$, and suppose that $a \in f_1$ and $b \in f_{2n-1}$, i.e.,
$a$ is the vertex that $f_1$ shares with $f_{2n}$, and $b$ is the vertex of $f_{2n-1}$
shared with $f_{2n}$.
By Theorem \ref{generatordescription} (2), a universal Gr\"obner 
basis of $I_H$ is given by the primitive binomials that 
correspond to even closed walks. {\color{blue} Consider a 
primitive closed even walk that passes through $f_1$.  It will have
one of the following forms:

\begin{enumerate}
    \item $(f_1,f_2,\ldots,f_{2n-1},e)$, or
    \item $(f_1,f_2,\ldots,f_{2n-1},e_{j_1},\ldots,e_{j_{2k-1}})$ for some odd walk $(e_{j_1},\ldots,e_{j_{2k-1}})$ in $G$ that connects the vertex $a$ of $f_1$ with the vertex $b$ of $f_{2n-1}$, or
    \item $(f_1,f_2,\ldots,f_{2n-1},e_{j_1},\ldots,e_{j_{2k-1}},f_{2n-1},f_{2n-2},\ldots,f_1,e_{i_1},\ldots,e_{i_{2r-1}})$ for some closed odd walk $(e_{j_1},\ldots,e_{j_{2k-1}})$ in $G$ that starts and ends at vertex $b$, and some closed odd walk $(e_{i_1},\ldots,e_{i_{2r-1}})$ in $G$ that starts and ends at vertex $a$.
\end{enumerate}

\noindent Thus, any primitive binomial involving the variable $f_1$ has the form:

\begin{enumerate}
    \item $f_1f_3\cdots f_{2n-1} - ef_2\cdots f_{2n-2}$, or
    \item $f_1f_3\cdots f_{2n-1}e_{j_2}e_{j_4}\cdots e_{j_{2k-2}} - 
f_2f_4\cdots f_{2n-2}e_{j_1}e_{j_3}\cdots e_{j_{2k-1}}$, or
\item {\small $f_1^2f_3^2\cdots f_{2n-1}^2e_{j_2}e_{j_4}\cdots e_{j_{2k-2}}e_{i_2}e_{i_4}\cdots e_{i_{2r-2}} - 
f_2^2f_4^2\cdots f_{2n-2}^2e_{j_1}e_{j_3}\cdots e_{j_{2k-1}}e_{i_1}e_{i_3}\cdots e_{i_{2r-1}}.$}
\end{enumerate}

Notice that for any $f_1$-compatible monomial order, the initial term of each binomial of type (2) or (3) can be divided by the initial term of the walk of type (1). That is, $f_1f_3\cdots f_{2n-1}$ divides the initial term of any binomial of type (2) or (3). Thus, all walks of type (2) and (3) are not part of a reduced Gr\"obner basis of $I_H$ with respect to this monomial order and are therefore not needed in the geometric vertex decomposition computation presented below.

Let $y = f_1$ and let $<$ be a $y$-compatible monomial order. Consider the reduced Gr\"obner basis of $I_H$, which by the above can be written
as $$\mathcal{G} =\{f_1f_3\cdots f_{2n-1} - ef_2\cdots f_{2n-2}, g_1,\ldots, g_r\}$$ where $y$ does not divide any term of $g_i$.} Each $g_1,\ldots,g_r$ corresponds to a primitive closed even walk that does not pass through $f_1$.  Consequently, each
$g_i$ corresponds to a primitive closed even walk in $G$.  Thus
$\langle g_1,\ldots,g_r \rangle = I_G$ (we abuse notation and write $I_G$ for the induced ideal $I_G\mathbb{K}[E(H)]$).

Additionally, by Lemma \ref{linktoricidealgraph} we have 
$N_{y,I_H} =\langle g_1, \ldots,g_r \rangle = I_{H\setminus f_1}$.
But note that in $H \setminus f_1$, the edge $f_2$ is a leaf.  Removing
$f_2$ from $(H \setminus f_1)$ makes $f_3$ a leaf, and so on.  So, by repeatedly
applying Lemma \ref{removeleaves}, we have 
$$N_{y,I_H} = \langle g_1,\ldots,g_r \rangle = I_{H\setminus f_1} = I_{(H\setminus f_1)\setminus f_2} =
\cdots = I_{(\cdots (H\setminus f_1) \cdots )\setminus f_{2n-1}} = I_G.$$

{\color{blue}
\noindent Similarly, since $f_1f_3\cdots f_{2n-1} - ef_2\cdots f_{2n-2}$ is the only element of $\mathcal{G}$ containing a term divisible by $y =f_1$, we must have 
$$C_{y,I_H} = \langle  f_3\cdots f_{2n-1}, g_1,\ldots,g_r\rangle 
= \langle f_3 \cdots f_{2n-1} \rangle + I_G.$$}
\noindent It is now straightforward to check that 
$${\rm in}_y(I_H) = \langle f_1f_3\cdots f_{2n-1} \rangle+ I_G = C_{y,I_H} \cap (N_{y,I_H} + \langle y \rangle),$$
thus giving a geometric vertex decomposition of $I_H$ with respect to $y$.  (We could also deduce this 
from Lemma \ref{square-freey} since each $d_i =1$ in 
our description of $\mathcal{G}$ above.)

To complete the proof, the contraction of $N_{y,I_H}$
to $\mathbb{K}[f_2,\ldots,f_{2n},e_1,\ldots,e_s]$ satisfies
$$N_{y,I_H} = \langle 0 \rangle + I_G \subseteq \mathbb{K}[f_2,\ldots,f_{2n}]
\otimes \mathbb{K}[E(G)].$$
So $N_{y,I_H}$ is geometrically vertex decomposable by Theorem \ref{tensorproduct}
since $I_G$ is geometrically vertex decomposable in $\mathbb{K}[E(G)]$, and similarly 
for $\langle 0 \rangle$ in $\mathbb{K}[f_2,\ldots,f_n]$.  The ideal
$C_{y,I_H}$ contracts to 
$$C_{y,I_H} = \langle f_3\cdots f_{2n-1} \rangle + I_G \subseteq 
\mathbb{K}[f_2,\ldots,f_{2n}] \otimes \mathbb{K}[E(G)].$$
Since $\langle f_3f_5\cdots f_{2n-1} \rangle \subseteq \mathbb{K}[f_2,\ldots,f_{2n}]$
is geometrically vertex decomposable by Lemma \ref{simplecases} (2), and $I_G$ is geometrically vertex decomposable in $\mathbb{K}[E(G)]$ by hypothesis,
the ideal
$C_{y,I_H}$ is geometrically vertex decomposable by again appealing to
Theorem \ref{tensorproduct}.   Thus $I_H$ is geometrically
vertex decomposable, as desired.
\end{proof}

\begin{example}
Let $G$ be a cycle of even length, i.e., $G$ has edge set $e_1,\ldots,e_{2n}$
with $(e_1,\ldots,e_{2n})$ a closed even walk. The ideal
$I_G = \langle e_1e_3 \cdots e_{2n-1} - e_2e_4\cdots e_{2n} \rangle$ is geometrically vertex decomposable
by Lemma \ref{simplecases} (2).
By repeatedly 
applying Theorem \ref{gluetheorem}, we can glue on even cycles to
make new graphs whose toric ideals are geometrically vertex decomposable.  Note that
by Lemma \ref{removeleaves}, we can also add leaves (and leaves to leaves, and so on) and not destroy
the geometrically vertex decomposability property.   
These constructions allow us to build many
graphs whose toric ideal is   geometrically vertex decomposable.

As a specific example,  the graph in Figure \ref{gluingexample} 
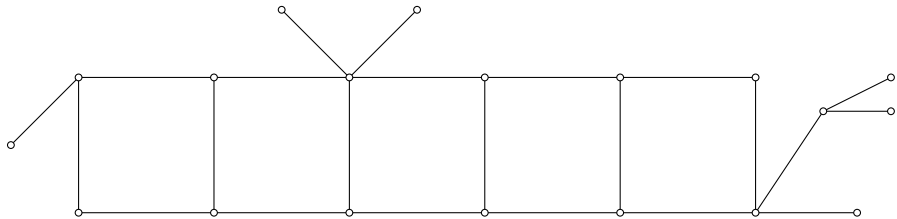
\begin{figure}[!ht]
    \centering
    \begin{tikzpicture}[scale=0.45]
      \draw (-4,0) -- (0,0);
      \draw (0,0) -- (4,0);
      \draw (4,0) -- (8,0);
      \draw (8,0) -- (12,0);
      \draw (12,0) -- (16,0);
       \draw (-4,4) -- (0,4);
      \draw (0,4) -- (4,4);
      \draw (4,4) -- (8,4);
      \draw (8,4) -- (12,4);
      \draw (12,4) -- (16,4);
      \draw (-4,0) -- (-4,4);
      \draw (0,0) -- (0,4);
      \draw (4,0) -- (4,4);
      \draw (8,0) -- (8,4);
      \draw (12,0) -- (12,4);
      \draw (16,0) -- (16,4);
      \draw (16,0) -- (18,3);
      \draw  (16,0) -- (19,0);
      \draw  (18,3) -- (20,4);
      \draw  (18,3) -- (20,3);
      \draw  (4,4) -- (2,6);
      \draw   (4,4) -- (6,6);
      \draw (-6,2) -- (-4,4);
   
      \fill[fill=white,draw=black] (-4,0) circle (.1);
       \fill[fill=white,draw=black] (0,0) circle (.1);
       \fill[fill=white,draw=black] (4,0) circle (.1);
       \fill[fill=white,draw=black] (8,0) circle (.1);
       \fill[fill=white,draw=black] (12,0) circle (.1);
       \fill[fill=white,draw=black] (16,0) circle (.1);
       \fill[fill=white,draw=black] (-4,4) circle (.1);
       \fill[fill=white,draw=black] (0,4) circle (.1);
       \fill[fill=white,draw=black] (4,4) circle (.1);
       \fill[fill=white,draw=black] (8,4) circle (.1);
       \fill[fill=white,draw=black] (12,4) circle (.1);
       \fill[fill=white,draw=black] (16,4) circle (.1);
       \fill[fill=white,draw=black] (18,3) circle (.1);
        \fill[fill=white,draw=black] (20,4) circle (.1);
  \fill[fill=white,draw=black] (20,3) circle (.1);
  \fill[fill=white,draw=black] (19,0) circle (.1);
   \fill[fill=white,draw=black] (2,6) circle (.1);
    \fill[fill=white,draw=black] (6,6) circle (.1);
     \fill[fill=white,draw=black] (-6,2) circle (.1);
    \end{tikzpicture}
    \caption{A graph whose toric ideal is geometrically vertex decomposable}
    \label{gluingexample}
 \end{figure}
is geometrically vertex decomposable
since we have repeatedly glued cycles of length four along edges,
and then added some leaves.  This
bipartite graph 
is also an example of what Gitler, Reyes, and
Villarreal call a {\it ring graph} \cite[Definition 2.5]{GRV}.
\end{example}

\section{Toric ideals of graphs and the glicci property}

In this section  we recall some of the basics of Gorenstein liaison (Section \ref{sec:Gliaison}) and then show that some large classes of toric ideals of graphs are glicci (Section \ref{sec:someGlicci}). 
This section is partly motivated by a result of Klein and Rajchgot \cite[Theorem 4.4]{KR}, which says that geometrically vertex decomposable ideals are 
glicci.   We note that while geometrically vertex decomposable ideals are glicci, glicci ideals need not be geometrically vertex decomposable. Indeed, we do not know if the toric ideals of graphs proven to be glicci in this section are also geometrically vertex decomposable.  However, the results
of this section make use of the geometric vertex decomposition language
of Remark \ref{gvdremark}. For the remainder of the section, we will let $S=\mathbb{K}[x_0,\ldots,x_n]$ denote the graded polynomial ring with respect to the standard grading.

\subsection{Gorenstein liaison preliminaries}\label{sec:Gliaison}

We provide a quick review of the basics of Gorenstein liaison;  our main
references for this material are \cite{MN1,MN}.

\begin{definition}
Suppose that $V_1,V_2,X$ are subschemes of $\mathbb{P}^n$ defined by saturated ideals $I_{V_1},I_{V_2}$ and $I_X$ of $S =\mathbb{K}[x_0,\ldots,x_n]$, respectively. Suppose also that $I_X\subset I_{V_1}\cap I_{V_2}$ and $I_{V_1}=I_X:I_{V_2}$ and $I_{V_2}=I_X:I_{V_1}$. We say that $V_1$ and $V_2$ are \textit{directly algebraically $G$-linked} 
if $X$ is Gorenstein. 
In this case we write
$V_1 \stackrel{X}{\sim} V_2$.
\end{definition}

We can now define an equivalence relation using the notion
of algebraically $G$-linked.

\begin{definition}
Let $V_1,\ldots,V_k$ be subschemes of $\mathbb{P}^n$ defined by the
saturated ideals $I_{V_1},\ldots,I_{V_k}$.  If there are 
Gorenstein varieties $X_1,\ldots,X_{k-1}$ such that 
$V_1 \stackrel{X_1}{\sim} V_2
\stackrel{X_2}{\sim} \cdots \stackrel{X_{k-1}}{\sim} V_k$,
then we say $V_1$ and $V_k$ are in the same {\it Gorenstein liaison class} (or
{\it $G$-liaison class}) and $V_1$ and $V_k$ are {\it $G$-linked}
in $k-1$ steps.   If $V_k$ is a complete intersection, then
we say $V_1$ is in the {\it Gorenstein liaison class of a complete 
intersection} or {\it glicci}.
\end{definition}

In what follows, we say a homogeneous saturated ideal $I$ is glicci if the subscheme defined by
$I$ is glicci.

\begin{example}
Consider the twisted cubic $V_1\subset \mathbb{P}^3$ with 
\[I_{V_1} = 
\langle xz-y^2, xw-z^2, xw-yz \rangle \subseteq \mathbb{K}[x,y,z].  \]
Choose two of these quadrics, and let $X$ be subscheme defined by their intersection. It is an exercise to check that $X$ is the union of $V_1$ and a line, which we denote by $V_2$. Therefore, $V_1 \stackrel{X}{\sim} V_2$. Furthermore, since $X$ is a complete intersection, and thus
Gorenstein, the twisted cubic and a line are directly $G$-linked.
\end{example}

\begin{remark}
One of the main open questions in liaison theory asks if every arithmetically
Cohen-Macaulay subscheme of $\mathbb{P}^n$ is glicci (see \cite[Question 1.6]{KMMNP}).
\end{remark}

While it can be difficult in general to find a sequence of $G$-links between two varieties, there is a tool called an elementary $G$-biliaison which simplifies the process when it exists. 

\begin{definition}
Let $S=\mathbb{K}[x_0,\ldots,x_n]$ with the standard grading. Let $C$ and $I$ be homogeneous, saturated, and unmixed ideals of $S$ such that ${\rm ht}(C)={\rm ht}(I)$. Suppose that there is some $d\in\mathbb{Z}$ and Cohen-Macaulay homogeneous ideal $N\subset C\cap I$ with ${\rm ht}(N)={\rm ht}(I)-1$ such that $I/N$ is
isomorphic to $[C/N](-d)$ as an $R/N$-module. If $N$ is generically Gorenstein, then $I$ is obtained from $C$ via an \textit{elementary $G$-biliaison of height $d$}.
\end{definition}

In fact, suppose that $V$ and $W$ are two subschemes of $\mathbb{P}^n$ such that $I_{V}$ and $I_{W}$ are homogeneous, saturated and unmixed ideals. If $I_{V}$ is obtained from  $I_{W}$ by an elementary $G$-biliaison, then $V$ and $W$ are $G$-linked in two steps \cite[Theorem 3.5]{Ha}.  Moreover, elementary
$G$-biliaisons preserve the Cohen-Macaulay property. This and other properties of $G$-linked varieties can be found in \cite{MN1}. Indeed, we will use the following:

\begin{lemma}\label{linkage_CM}\cite[Corollary 5.13]{MN1}
Let $I$ and $J$ be homogeneous, saturated ideals in $S$ and assume that $I$ and $J$ are directly $G$-linked. Then $S/I$ is Cohen-Macaulay if and only $S/J$ is Cohen-Macaulay.
\end{lemma}

Migliore and Nagel have given a criterion for an ideal to be glicci.

\begin{theorem}\cite[Lemma 2.1]{MN}\label{MN_glicci}
Let $I\subset S$ be a homogeneous ideal such that $S/I$ is Cohen-Macaulay and
generically Gorenstein. If $f\in S$ is a homogeneous non-zero-divisor of $S/I$, then the ideal $I + \langle f \rangle \subset S$ is glicci.
\end{theorem}

Another criterion for an ideal to be glicci is geometric vertex decomposability. In fact a geometric vertex decomposition gives rise to an elementary $G$-biliaison of height 1.

\begin{lemma}\cite[Corollary 4.3]{KR}\label{GVD_linkage}
Let $I$ be a homogeneous, saturated, unmixed ideal of $S$ and $\text{in}_y I = C_{y,I} \cap (N_{y,I}+\langle y \rangle)$ a nondegenerate geometric vertex decomposition with respect to some variable $y = x_i$ of $S$.  Assume that $N_{y,I}$ is Cohen--Macaulay and generically Gorenstein and that $C_{y,I}$ is also unmixed.  Then $I$ is obtained from $C_{y,I}$ by an elementary $G$-biliaison of height $1$.
\end{lemma}

\begin{theorem}\cite[Theorem 4.4]{KR}\label{gvd=>glicci} If the saturated homogeneous ideal $I \subseteq S$ is geometrically vertex
decomposable, then $I$ is glicci.
\end{theorem}

As noted in the introduction of the paper, the previous result partially motivates
our interest in developing a deeper understanding of
geometrically vertex decomposable ideals.

\subsection{Some toric ideals of graphs which are glicci}\label{sec:someGlicci}

In this section we use Migliore and Nagel's result \cite[Lemma 2.1]{MN} (see Theorem \ref{MN_glicci} above) to show that some classes of toric ideals of graphs are glicci. We begin with a straightforward consequence of this theorem together with \cite[Theorem 3.7]{FHKVT}.

\begin{theorem}\label{glue=glicci}
Let $G$ be a finite simple graph such that $\mathbb{K}[E(G)]/I_G$ is Cohen-Macaulay. Let $H$ be the graph obtained by gluing an even cycle $C$ to $G$ along any edge. Then $I_H$ is glicci. 
\end{theorem}

\begin{proof}
As in the proof of Theorem \ref{gluetheorem}, let $E(G) = \{e_1,\ldots, e_s\}$ denote the edges of $G$ and $E(C) = \{f_1,\ldots, f_{2n}\}$ denote the (consecutive) edges of the even cycle $C$. Assume that $C$ is glued to $G$ along $f_{2n}$ and $e$. Then $\mathbb{K}[E(H)] = \mathbb{K}[E(G)]\otimes \mathbb{K}[f_1,\dots,f_{2n-1}]$. For convenience, write $I_G$ for the induced ideal $I_G\mathbb{K}[E(H)]$.

Let $F = f_1f_3\cdots f_{2n-1}-f_2f_4\cdots f_{2n-2}e$ be the primitive binomial associated to the even cycle $C$. By \cite[Theorem 3.7]{FHKVT}, $I_H = I_G+\langle F\rangle$. As $I_G$ is prime, we have that $F$ is a homogeneous non-zero-divisor on $\mathbb{K}[E(H)]/I_G$ and $\mathbb{K}[E(H)]/I_G$ is generically Gorenstein. As $\mathbb{K}[E(H)]/I_G$ is Cohen-Macaulay by assumption, Theorem \ref{MN_glicci} implies that $I_H$ is glicci.
\end{proof}

We can combine a one step geometric vertex decomposition with Theorem \ref{MN_glicci} to see that many toric ideals of graphs which contain $4$-cycles are glicci. Our main theorem in this direction is Theorem \ref{thm:gapFreeGlicci}, which says that the toric ideal of a \emph{gap-free} graph containing a $4$-cycle is glicci. We begin with a general lemma which is not necessarily about toric ideals of graphs.

 \begin{lemma}\label{MN-combined}
 Let $S =\mathbb{K}[x_0,\ldots, x_n]$ with the standard grading, and $I\subset S$ be a homogeneous, saturated ideal such that $S/I$ is Cohen-Macaulay. Assume the following conditions are satisfied:
 \begin{enumerate}
 \item $I$ is square-free in $y$ with a nondegenerate geometric vertex decomposition \[\init_y(I) = C_{y,I}\cap (N_{y,I}+\langle y \rangle);\] 
 \item $I$ contains a homogeneous polynomial $Q$ of degree $2$ such that $y$ divides some term of $Q$; and
 \item $S/N_{y,I}$ is Cohen-Macaulay and generically Gorenstein, and $C_{y,I}$ is radical. 
 \end{enumerate}
 Then $I$ is glicci.
 \end{lemma}
 
 \begin{proof}
By assumption (1), we have a nondegenerate geometric vertex decomposition $\init_y(I) = C_{y,I}\cap (N_{y,I}+\langle y \rangle)$. 
 Since $I$ is Cohen-Macaulay and hence unmixed, we can conclude that $C_{y,I}$ is equidimensional by \cite[Lemma 2.8]{KR}. Since $C_{y,I}$ is also radical by assumption (3), $C_{y,I}$ must be unmixed. Furthermore, because $S/N_{y,I}$ is Cohen-Macaulay and generically Gorenstein by assumption (3), we may use Lemma  \ref{GVD_linkage} to see that the geometric vertex decomposition gives rise to an elementary $G$-biliaison of height 1 from $I$ to $C_{y,I}$. Hence $S/I$ being  Cohen-Macaulay implies that $S/C_{y,I}$ is too by Lemma \ref{linkage_CM}. 
 
 Let $<$ be a $y$-compatible monomial order. By assumptions (1) and (2), $I$ contains a degree $2$ form which can be written as $Q = yf+R$ where $y$ does not divide any term in $f$ or $R$.
 Thus, $f\in C_{y,I}$. 
 Let $z = \init_<(f)$. 
 Since the geometric vertex decomposition $\init_y(I) = C_{y,I}\cap (N_{y,I}+\langle y \rangle)$ is nondegenerate, we have that $C_{y,I}\neq \langle 1\rangle$. Hence $C_{y,I}$ has a reduced Gr\"obner basis of the form
 $\{f',t_1,\dots, t_s\}$ where $\init_<(f') = z$ and $z$ does not divide any term of any $t_i$, $1\leq i\leq s$. Let $C' = \langle t_1,\dots, t_s\rangle$ so that $C_{y,I} = \langle f'\rangle+ C'$. With this set-up, we see that $f'\neq 0$ is a non-zero-divisor on $S/C'$.
 
Let $S_{\hat{z}} = \mathbb{K}[x_1,\dots,\hat{z},\dots, x_n]$. Then $S/C_{y,I}\cong S_{\hat{z}}/ C'$. 
Thus, $S_{\hat{z}}/C'$ (and hence $S/C'$ after extending $C'$ to $S$) is Cohen-Macaulay because $S/C_{y,I}$ is Cohen-Macaulay. Similarly, $C_{y,I}$ being radical implies that $C'$ (viewed in $S_{\hat{z}}$ or $S$) is radical. Thus, by \cite[Lemma 2.1]{MN} (see Theorem \ref{MN_glicci}), we conclude that $C_{y,I}$ is glicci. 
  
 By applying the elementary $G$-biliaison between $I$ and $C_{y,I}$ once more, we conclude that $I$ is also glicci.
 \end{proof}

We will now apply Lemma \ref{MN-combined} to see that certain classes of toric ideals of graphs are glicci. In what follows, let $y = x_i$ be an indeterminate in $S = \mathbb{K}[x_1,\dots, x_n]$ and let $<$ be a $y$-compatible monomial order. Let $M^G_{y}$ be the ideal generated by all monomials $m\in S$ such that $ym-r\in \mathcal{U}(I_G)$ and $\init_< (ym-r) = ym$. Observe that $M^G_{y}$ does not depend on the choice of $y$-compatible monomial order. Furthermore, since $I_G$ is prime and $ym-r$ is primitive, $y$ cannot appear in both terms of the binomial. We will consider generalizations of $M^G_y$ in Section \ref{section_square-free}.

\begin{theorem}\label{prop:glicciGraphs}
Let $G$ be a finite simple graph where $\mathbb{K}[E(G)]/I_G$ is Cohen-Macaulay. 
Suppose that there exists an edge $y\in E(G)$ such that $y$ is contained in a 4-cycle of $G$, and a $y$-compatible monomial order $<_y$ such that $\init_y(I_G)$ is square-free in $y$. 
Suppose also that $I_{G\setminus y}$ is Cohen-Macaulay 
and $I_{G\setminus y}+M^G_{y}$ is radical. 
Then $I_G$ is glicci.
\end{theorem}
\begin{proof}
We will show that the three assumptions of Lemma \ref{MN-combined} hold. Let $<$ be a $y$-compatible monomial order.

Since $I_G$ is square-free in $y$, there exists a geometric vertex decomposition 
\[\init_y(I_G) = C_{y,I_G}\cap (N_{y,I_G}+\langle y \rangle)\] by  Lemma \ref{square-freey}. Then $N_{y,I_G} = I_{G\setminus y}$ and $C_{y,I_G}=I_{G\setminus y}+M^G_{y}$. 
Since $I_G$ is a toric ideal of a graph, and  hence generated in degree $2$ or higher, we do not have that $C_{y,I} = \langle 1\rangle$. Furthermore, $I_G$ and $N_{y,I_G}$ are each the toric ideal of a graph, hence radical (and therefore saturated since $I_G$ is not the irrelevant ideal), and $C_{y,I_G}$ is radical by assumption. Thus, by \cite[Proposition 2.4]{KR}, we conclude that the geometric vertex decomposition $\init_y(I_G) = C_{y,I_G}\cap (N_{y,I_G}+\langle y \rangle)$ is nondegenerate since the reduced Gr\"obner basis of $I_G$ involves $y$ by assumption. Thus, assumption (1) of Lemma \ref{MN-combined} holds.

Assumption (2) of Lemma \ref{MN-combined} holds because there exists an edge $y\in E(G)$ such that $y$ is contained in a 4-cycle of $G$. Assumption (3) of Lemma \ref{MN-combined} holds by the assumption that $I_{G\setminus y}$ is Cohen-Macaulay and $I_{G\setminus y}+M^G_{y}$ is radical.
\end{proof}

Recall from Theorem \ref{sqfree=>cm} that if $I_G\subseteq \mathbb{K}[E(G)]$ is a toric ideal of a graph which has a square-free degeneration, then $\mathbb{K}[E(G)]/I_G$ is Cohen-Macaulay. We can use Theorem \ref{prop:glicciGraphs} to show that many toric ideals of graphs which have both square-free degenerations and $4$-cycles are glicci. Specifically, we have the following: 

\begin{corollary}\label{cor:square-freeDegenGlicciGraph}
Let $G$ be a finite simple graph and suppose that there exists an edge $y\in E(G)$ such that $y$ is contained in a $4$-cycle of $G$. Suppose also that there exists some $y$-compatible monomial order $<$ such that $\init_< (I_G)$ is a square-free monomial ideal. Then $I_G$ is glicci. 
\end{corollary}

\begin{proof}
Since $\init_< (I_G)$ is a square-free monomial ideal, we have that $\mathbb{K}[E(G)]/I_G$ is Cohen-Macaulay. 
Furthermore, $I_G$ is square-free in $y$.

Let $\{yq_1+r_1,\dots, yq_s+ r_s, h_1,\dots, h_t\}$ be a reduced Gr\"obner basis for $I_G$ so that each $\init_< (yq_i)$, $1\leq i\leq s$, and each $\init_< (h_j)$, $1\leq j\leq t$ are square-free monomials. 
Consider the geometric vertex decomposition
\[
\init_y (I_G) = C_{y,I_G}\cap(N_{y,I_G}+\langle y\rangle).
\]
By \cite[Theorem 2.1]{KMY}, $\{h_1,\dots, h_t\}$ and $\{q_1,\dots, q_s, h_1,\dots, h_t\}$ are a Gr\"obner bases for $N_{y,I_G}$ and $C_{y,I_G}$ respectively. Thus, $\init_< (N_{y,I_G})$ and $\init_< (C_{y,I_G})$ are square-free monomial ideals. Since $N_{y,I_G} = I_{G\setminus y}$ is a toric ideal of a graph, it follows that $I_{G\setminus y}$ is Cohen-Macaulay. Since $C_{y,I_G} = I_{G\setminus y}+M^G_y$, it follows that $I_{G\setminus y}+M^G_y$ is radical. Thus, the assumptions of Theorem \ref{prop:glicciGraphs} hold and we conclude that $I_G$ is glicci.
\end{proof}

We end by proving that the toric ideal of a gap-free graph containing a $4$-cycle is glicci.
A graph $G$ is {\it gap-free} if for any two edges
$e_1 = \{u,v\}$ and $e_2 = \{w,x\}$ with $\{u,v\} \cap \{w,x\}
= \emptyset$, there is an edge $e \in E(G)$ that is
is adjacent to both $e_1$ and $e_2$, i.e., one of the edges
$\{u,w\}, \{u,x\}, \{v,w\}, \{v,x\}$ is also in $G$.  Note that
the name for this family is not standardized;  these 
graphs are sometimes called $2K_2$-free, or $C_4$-free, among other
names (see D'Al\`i \cite{DAli} for more).
Note that $G$ has a $4$-cycle if and only if the graph complement $\bar{G}$ is not gap-free. 

\begin{theorem}\label{thm:gapFreeGlicci}
Let $G$ be a gap-free graph such that the graph complement $\bar{G}$ is not gap-free. Then $I_G$ is glicci.
\end{theorem}
\begin{proof}
Since $\bar{G}$ is not gap-free, $G$ must contain a 4-cycle. Pick any variable $y$ belonging to this cycle. By \cite[Theorem 3.9]{DAli}, since $G$ is gap-free, there exists a $y$-compatible order $<_y$ such that $\init_{<_y}(I_G)$ is square-free (we can ensure this by choosing $<_\sigma$ in \cite[Theorem 3.9]{DAli} so that the vertices defining $y$ have the smallest weight). The result now follows from Corollary \ref{cor:square-freeDegenGlicciGraph}.
\end{proof}

\section{Toric ideals of bipartite graphs}\label{sec:bipartite}

In this section, we show that toric ideals of bipartite graphs are geometrically vertex decomposable. In Section \ref{sect:generalCase}, we treat the general case, making use of results of Constantinescu and Gorla from \cite{CG}.  
Then, in Section \ref{sect:specialCases} we give alternate proofs of geometric vertex decomposibility in special cases.

\subsection{Toric ideals of bipartite graphs are geometrically vertex decomposable}\label{sect:generalCase}

Recall that a simple graph $G$ is \emph{bipartite} if its vertex set $V(G) = V_1\sqcup V_2$ is a disjoint union of two sets $V_1$ and $V_2$, such that every edge in $G$ has one of its endpoints in $V_1$ and the other endpoint in $V_2$. The purpose of this subsection is to prove Theorem \ref{thm: gvdBipartite} below, which says that the toric ideal of a bipartite graph is geometrically vertex decomposable. We will make use of the results and ideas in Constaninescu and Gorla's paper \cite{CG} on Gorenstein liaison of toric ideals of bipartite graphs.

Let $G$ be a bipartite graph. Following \cite[Definition 2.2]{CG}, we say that a subset ${\bf e} = \{e_1,\dots, e_r\}\subseteq E(G)$ is a \emph{path ordered matching} of length $r$ if the vertices of $G$ can be relabelled such that $e_i = \{i,i+r\}$ and
\begin{enumerate}
    \item $f_i = \{i,i+r+1\}\in E(G)$, for each $1\leq i\leq r-1$, 
    \item if $\{i,j+r\}\in E(G)$ and $1\leq i,j \leq r$, then $i\leq j$.
\end{enumerate}
The following is straightforward. It will be referenced later in the subsection.

\begin{lemma}\label{lem:techLemma2}
Let ${\bf e} = \{e_1,\dots, e_r\}$ be a path ordered matching. Then $\{e_1,\dots, e_{r-1}\}$ is a path ordered matching on $G\setminus e_r$.
\end{lemma}

Given a subset ${\bf e} \subseteq E(G)$,
let $M^G_{\bf e}$ be the set of all monomials $m$ such that there is some non-empty subset ${\bf \tilde{e}}\subseteq {\bf e}$ where $m\left(\prod_{e_i\in {\bf \tilde{e}}}e_i\right)-n$ is the binomial associated to a cycle in $G$. 
Let
\begin{equation}\label{eq:IGe}
I^G_{\bf e} = I_{G\setminus {\bf e}}+\langle M^G_{\bf e}\rangle,  
\end{equation}
and observe that when $\bf{e} = \emptyset$, $I^G_{\bf e} = I_G$. 

Let $G$ be a bipartite graph and ${\bf e} = \{e_1,\dots, e_r\}$ a path ordered matching. 
Let $\prec$ be a lexicographic monomial order on $\mathbb{K}[E(G)]$ with $e_r>e_{r-1}>\cdots>e_1$ and $e_1>f$ for all $f\in E(G)\setminus {\bf e}$. 
Let $\mathcal{C}(G)$ denote the set of binomials associated to cycles in $G$. By \cite[Lemma 2.6]{CG}, $\mathcal{C}(G\setminus {\bf e})\cup M^G_{\bf e}$ is a Gr\"obner basis for $I^G_{\bf e}$ with respect to the term order $\prec$, and $\init_{\prec} (I^G_{\bf e})$ is a square-free monomial ideal. 

\begin{remark}\label{rmk:gbBipartite}
Let $\widetilde{M}^G_{\bf e}$ be the set of monomials $m$ such that there is some non-empty subset ${\bf \tilde{e}}\subseteq {\bf e}$ where $m\left(\prod_{e_i\in {\bf \tilde{e}}}e_i\right)-n$ is the binomial associated to a cycle in $G$ and $n$ is not divisible by any $e_i\in {\bf e}$. 
By \cite[Remark 2.7]{CG}, $\mathcal{C}(G\setminus {\bf e})\cup \widetilde{M}^G_{\bf e}$ is also a Gr\"obner basis for $I^G_{\bf e}$ with respect to $\prec$. Furthermore, observe that if $me_i\in \widetilde{M}^G_{\bf e}$ for some $e_i\in {\bf e}$, then $m$ is also an element of $\widetilde{M}^G_{\bf e}$. Hence, if we let $L^G_{\bf e}$ be the set of monomials in $\widetilde{M}^G_{\bf e}$ which are not divisible by any $e_i\in {\bf e}$, then $\mathcal{C}(G\setminus {\bf e})\cup L^G_{\bf e}$ is a Gr\"obner basis for $I^G_{\bf e}$ with respect to $\prec$.
\end{remark}

Using Remark \ref{rmk:gbBipartite}, we obtain the following lemma, which we will need when proving geometric vertex decomposability of toric ideals of bipartite graphs. 

\begin{lemma}\label{lem:gbBipartite2}
Let $G$ be a bipartite graph and let ${\bf e} = \{e_1,\dots, e_r\}$, $r\geq 1$, be a path ordered matching on $G$, and let ${\bf e'} = \{e_1,\dots, e_{r-1}\}$. Let $\prec$ be a lexicographic monomial order on $\mathbb{K}[E(G)]$ with $e_r>e_{r-1}>\cdots>e_1$ and $e_1>f$ for all $f\in E(G)\setminus {\bf e}$. 
The set $\mathcal{C}(G\setminus {\bf e'})\cup L^G_{\bf e'}$ is a Gr\"obner basis for $I^G_{\bf e'}$ with respect to $\prec$ and $\init_{\prec}(I^G_{\bf e'})$ is a square-free monomial ideal.
\end{lemma}

\begin{proof}
By Remark \ref{rmk:gbBipartite}, $\mathcal{G}:=\mathcal{C}(G\setminus {\bf e'})\cup L^G_{\bf e'}$ is a Gr\"obner basis for $I^G_{\bf e'}$ with respect to the lexicographic term order $e_{r-1}>e_{r-2}>\cdots >e_1>e_r$ and $e_r>f$ for all $f\in E(G)\setminus {\bf e}$. 
Since none of $e_1,\dots, e_{r-1}$ appear in $\mathcal{G}$, we have that $\mathcal{G}$ is also a Gr\"obner basis for the lexicographic monomial order $\prec$. Furthermore, all terms of all elements in $\mathcal{G}$ are square-free, so $\init_{\prec} (I^G_{\bf e'})$ is a square-free monomial ideal.
\end{proof}

We now use Lemma \ref{lem:gbBipartite2} to obtain a geometric vertex decomposition of $I^G_{\bf e'}$:

\begin{proposition}\label{lem:bipartiteGVDgeneral}
Let $G$ be a bipartite graph and let ${\bf e} = \{e_1,\dots, e_r\}$ be a path ordered matching. Let ${\bf e'} = \{e_1,\dots, e_{r-1}\}$. Then there is a geometric vertex decomposition
\begin{equation}\label{eq:gvdEquation}
\init_{e_r}(I^G_{{\bf e'}}) = (I^{G\setminus e_r}_{\bf e'}+\langle e_r\rangle)\cap I^G_{\bf e}.
\end{equation}
\end{proposition}

\begin{proof}
Let $\prec$ be a lexicographic monomial order on $\mathbb{K}[E(G)]$ with $e_r>e_{r-1}>\cdots >e_1$ and $e_1>f$ for all $f\in E(G)\setminus {\bf e}$. This is an $e_r$-compatible monomial order. 
By Lemma \ref{lem:gbBipartite2}, 
$\mathcal{C}(G\setminus {\bf e'})\cup L^G_{\bf e'}$ is a Gr\"obner basis for $I^G_{\bf e'}$ with respect to $\prec$, and $\mathcal{C}(G\setminus {\bf e'})\cup L^G_{\bf e'}$ are square-free in $e_r$. We can write:
\[
\mathcal{C}(G\setminus {\bf e'}) = \{e_rm_1-n_1, e_rm_2-n_2,\dots, e_rm_q-n_q, h_1,\dots, h_t\}, \text{ and }
\]
\[
L^G_{\bf e'} = \{e_ra_1,\dots, e_ra_u, b_1,\dots, b_v\}
\]
where $e_r$ does divide any $m_\ell$, $n_\ell$, $1\leq \ell\leq q$, nor any term of $h_k$, $1\leq k\leq t$, nor any of the monomials $a_1,\dots, a_u, b_1,\dots, b_v$. We thus have the geometric vertex decomposition
\begin{align*}
\text{in}_{{e}_r}(I^G_{\bf e'}) &= (\langle h_1,\dots, h_t, b_1,\dots, b_v\rangle + \langle {e}_r\rangle) \cap \langle m_1,\dots, m_q, h_1,\dots, h_t, a_1,\dots, a_u, b_1,\dots, b_v\rangle\\
&= (\langle h_1,\dots, h_t, b_1,\dots, b_v\rangle + \langle {e}_r\rangle) \cap I^G_{\bf e}.
\end{align*}
It remains to show that $\langle h_1,\dots, h_t, b_1,\dots, b_v\rangle = I^{G\setminus e_r}_{\bf e'}$. 

By Lemma \ref{lem:techLemma2}, ${\bf e'}$ is a path ordered matching on $G\setminus e_r$. 
Thus, $I^{G\setminus e_r}_{\bf e'}$ is generated by \[\mathcal{C}((G\setminus e_r)\setminus {\bf e'})\cup L^{G\setminus e_r}_{\bf e'} = \mathcal{C}(G\setminus {\bf e})\cup L^{G\setminus e_r}_{\bf e'}.\]
Observe that $\{h_1,\dots, h_t \} = \mathcal{C}(G\setminus {\bf e})$. 
Also, it follows from the definitions that $L^{G\setminus e_r}_{\bf e'}\subseteq \{b_1,\dots, b_v\}$. Thus, we have the inclusion $I^{G\setminus e_r}_{\bf e'}\subseteq \langle h_1,\dots, h_t, b_1,\dots, b_v\rangle$.

For the reverse inclusion, fix some $b_j$, $1\leq j\leq v$. Then there is some non-empty subset ${\widetilde{\bf e}}\subseteq {\bf e'}$ and a binomial $b_j(\prod_{e_i\in {\widetilde{\bf e}}}e_i)-n$ associated to a cycle in $G$. 
If $e_r$ does not divide $n$ then $b_j\in M^{G\setminus e_r}_{\bf e'}$, and hence $b_j\in I^{G\setminus e_r}_{\bf e'}$. Otherwise, since ${\bf e}$ is also a path ordered matching, one can apply the proof of \cite[Remark 2.7]{CG} to find another cycle in $G$ which does not pass through $e_r$ and which gives rise to an element $c_j\in M^G_{\bf e}$ which divides $b_j$. Since the cycle does not pass through $e_r$, we have $c_j\in M^{G\setminus e_r}_{\bf e'}$. As $\mathcal{C}((G\setminus e_r)\setminus {\bf e'})\cup M^{G\setminus e_r}_{\bf e'}$ is a Gr\"obner basis for $I^{G\setminus e_r}_{\bf e'}$, we see that $c_j$, and hence $b_j$, is an element of $I^{G\setminus e_r}_{\bf e'}$. Thus, $\langle h_1,\dots, h_t, b_1,\dots, b_v\rangle\subseteq I^{G\setminus e_r}_{\bf e'}$.
\end{proof}

We say that a path ordered matching ${\bf e} = \{e_1,\dots, e_r\}$ is \emph{right-extendable} if there is some edge $e_{r+1}\in E(G)$ such that $\{e_1,\dots,e_r,e_{r+1}\}$ is also a path ordered matching.

\begin{lemma}\label{lem:techLemma1}
Let $G$ be a bipartite graph with no leaves and let ${\bf e} = \{e_1,\dots, e_r\}$ be a path ordered matching which is not right-extendable. Then, $M^G_{\bf e}$ contains an indeterminate $x\in E(G)$ and ${\bf e}$ is a path ordered matching on $G\setminus x$. Furthermore, $I^G_{\bf e} = I^{G\setminus x}_{\bf e}+ \langle x\rangle$.
\end{lemma}

\begin{proof}
The proof is identical to the proof of \cite[Lemmas 2.12 and 2.13]{CG} upon replacing maximal path ordered matchings in \cite[Lemmas 2.12 and 2.13]{CG} with right-extendable path ordered matchings.
\end{proof}

\begin{lemma}\label{lem:techLemma3}
Let $G$ be a bipartite graph and let ${\bf e} = \{e_1,\dots, e_r\}$ be a path ordered matching. 
Suppose that $G$ has a leaf $y$. Then:
\begin{enumerate}
    \item if $y\notin {\bf e}$, then ${\bf e}$ is a path ordered matching in $G\setminus y$ and $I^G_{\bf e} = I^{G\setminus y}_{\bf e}$;
    \item if $y\in {\bf e}$, then $y = e_1$ or $e_r$ and ${\bf e}\setminus y$ is a path ordered matching in $G\setminus y$ and  $I^G_{\bf e} = I^{G\setminus y}_{{\bf e}\setminus y}$.
\end{enumerate}
\end{lemma}

\begin{proof}
Since ${\bf e}$ is a path ordered matching, the vertices of $G$ can be labelled such that $e_i = \{i,i+r\}$, $1\leq i\leq r$. Let $f_i = \{i,i+r+1\}$, $1\leq i\leq r-1$ so that $$e_1,f_1,e_2,f_2,\dots,e_{r-1},f_{r-1}, e_r$$ is a path of consecutive edges in $G$. 
Since $y$ is a leaf, we see that $y\notin \{f_1,\dots, f_{r-1}\}$. If $y\notin {\bf e}$, then each $e_i, f_i$ remains and ${\bf e}$ is a path ordered matching in $G\setminus y$. 
Furthermore no cycle in $G$ passes through $y$, hence $I^G_{\bf e} = I^{G\setminus y}_{\bf e}$.

If $y\in {\bf e}$, then either $y = e_1$ or $y = e_r$. In either case, since each $f_i$ remains in $G\setminus y$, ${\bf e}\setminus y$ is still a path ordered matching in $G\setminus y$.
Since there is no cycle in $G$ which passes through $y$, we have $I^{G}_{\bf e} = I^G_{{\bf e}\setminus y} = I^{G\setminus y}_{{\bf e}\setminus y}$.
\end{proof}

We will need one more result from \cite{CG}: 

\begin{theorem}\cite[Theorem 2.8]{CG}\label{bipartite_CM}
Let $G$ be a bipartite graph and ${\bf e} = \{e_1,\dots, e_r\}$ a path ordered matching. Then $\mathbb{K}[E(G)]/I^G_{\bf e}$ is Cohen-Macaulay.
\end{theorem}

We now adapt the proof of \cite[Corollary 2.15]{CG} on vertex decomposability of the simplicial complex associated to an initial ideal of $I^G_{\bf e}$ to prove the main theorem of this subsection.

\begin{theorem}\label{thm: gvdBipartite}
Let $G$ be a bipartite graph and ${\bf e} = \{e_1,\dots, e_r\}$ a path ordered matching. Then the ideal $I^G_{\bf e}$ is geometrically vertex decomposable. In particular, the toric ideal $I_G$ is geometrically vertex decomposable.
\end{theorem}

\begin{proof}
Let $R = \mathbb{K}[E(G)]$. By Theorem \ref{bipartite_CM}, each $R/I^G_{\bf e}$ is Cohen-Macaulay, hence unmixed.

We proceed by double induction on $|E(G)|$ and $s-r$ where ${\bf {\tilde{e}}} = \{{\tilde{e}}_1,\dots, \tilde{e}_s\}$ is a path ordered matching that is not right-extendable and is such that $\tilde{e}_1 = e_1,\dots, \tilde{e}_r = e_r$.  

If $|E(G)|\leq 3$, then $I_G = \langle 0\rangle$ as there are no primitive closed even walks in $G$, so  the result holds trivially. 

If $G$ has a leaf, then by Lemma \ref{lem:techLemma3}, there is an edge $y$ and a path ordered matching ${\bf e'}$ in $G\setminus y$ such that $I^G_{\bf e} = I^{G\setminus y}_{{\bf e'}}$. By induction on the number of edges in the graph, $I^{G\setminus y}_{{\bf e'}}$ is geometrically vertex decomposable, hence so is $I^G_{\bf e}$. 

So, assume that $G$ has no leaves. If $s-r = 0$, then ${\bf e}$ is not right extendable. 
Then, by Lemma \ref{lem:techLemma1}, there is an indeterminate $z\in M^G_{\bf e}$ such that
\[
I^G_{\bf e} = I^{G\setminus z}_{\bf e}+\langle z\rangle.
\]
By Lemma \ref{lem:techLemma1}, {\bf e} is a path ordered matching on $G\setminus z$, so again by induction on the number of edges in the graph, we have the $I^{G\setminus z}_{\bf e}$ is geometrically vertex decomposable, hence so is $I^G_{\bf e}$.

Now suppose that ${\bf e}$ is right extendable, so that $s-r>0$ and ${\bf e^*} = \{e_1,\dots, e_{r+1}\}$ is a path ordered matching. By Lemma \ref{lem:bipartiteGVDgeneral}, we have the geometric vertex decomposition
\begin{equation*}
\init_{e_{r+1}}(I^G_{{\bf e}}) = (I^{G\setminus e_{r+1}}_{\bf e}+\langle e_{r+1}\rangle)\cap I^G_{{\bf e^*}}.
\end{equation*}
By Lemma \ref{lem:techLemma2}, ${\bf e}$ is a path ordered matching on $G\setminus e_{r+1}$. So, by induction on the number of edges, $I^{G\setminus e_{r+1}}_{\bf e}$ is geometrically vertex decomposable. By induction on $s-r$, $I^G_{{\bf e^*}}$ is geometrically vertex decomposable. Hence, $I^G_{\bf e}$ is geometrically vertex decomposable.

The final conclusion now
follows from the fact that $I_G = I_{\bf e}^G$ when
${\bf e} = \emptyset$.
\end{proof}

\subsection{Alternate proofs in special cases}\label{sect:specialCases}
In this section, we apply results from the literature to give alternate proofs of geometric vertex decomposability for some well-studied families of bipartite graphs.  These
proofs illustrate that in some cases, we 
can prove that a family of ideals is geometrically
vertex decomposable directly from the definition.
Moreover, these examples do not require the full strength
of the machinery of
Section 5.1; in particular, these families of examples have
the property that the ideals $C_{y,I}$ and $N_{y,I}$ usually
do  not leave the family of ideals we are considering, thus
giving us nice inductive proofs.

We define the relevant families
of graphs.  A {\it Ferrers graph} is 
a bipartite graph on the vertex set
$X = \{x_1,\ldots,x_n\}$ and $Y= \{y_1,\ldots,y_m\}$ such that 
$\{x_n,y_1\}$ and $\{x_1,y_m\}$ are edges, and if $\{x_i,y_j\}$ is an edge, then so 
are all the edges $\{x_k,y_l\}$ with
$1 \leq k \leq i$ and $1 \leq l \leq j$.
We associate a partition 
$\lambda = (\lambda_1,\lambda_2,\ldots,\lambda_n)$ with $\lambda_1 \geq \lambda_2 
\geq \cdots \geq \lambda_n$ to a Ferrers graph where $\lambda_i =
\deg x_i$.   Some of the properties of
the toric ideals of these graphs
were studied by Corso and Nagel \cite{CN}.  Following Corso and Nagel, we denote
a Ferrers graph as $T_\lambda$ where $\lambda$
denotes the associated partition.

As an example, consider the 
partition $\lambda = (5,3,2,1)$ which can be 
visualized as
\[
\begin{tabular}{cccccccc}
& $y_1$ &  $y_2$ &  $y_3$  & $y_4$ & $y_5$\\
$x_1$ & $\bullet$ & $\bullet$ & $\bullet$ & $\bullet$ & $\bullet$ \\
$x_2$ & $\bullet$ & $\bullet$ & $\bullet$ &  &  \\
$x_3$ & $\bullet$ & $\bullet$ &  & &\\
$x_4$ & $\bullet$ & & & & \\
\end{tabular}
\]
We have labelled the rows with the $x_i$ 
vertices and the columns with the $y_j$ vertices.  From this representation,
the graph $T_\lambda$ is the graph
on the vertex set $\{x_1,\ldots,x_4,y_1,\ldots,y_5\}$ where
$\{x_i,y_j\}$ is an edge if and only if 
there is dot in the row and column indexed
by $x_i$ and $y_j$ respectively.  Figure \ref{fig_ferrers} gives the corresponding
bipartite graph $T_\lambda$ for $\lambda$.
\begin{figure}[!ht]
    \centering
    \begin{tikzpicture}[scale=0.45]
      \draw (0,0) -- (0,5);
      \draw (5,0) -- (0,5);
       \draw (10,0) -- (0,5);
      \draw (15,0) -- (0,5);
       \draw (20,0) -- (0,5);
       \draw (0,0) -- (5,5);
      \draw (5,0) -- (5,5);
       \draw (10,0) -- (5,5);
      \draw (0,0) -- (10,5);
       \draw (5,0) -- (10,5);
       \draw (0,0) -- (15,5);
      
      \fill[fill=white,draw=black] (0,0) circle (.1) node[below]{$y_1$};
       \fill[fill=white,draw=black] (5,0) circle (.1) node[below]{$y_2$};
       \fill[fill=white,draw=black] (10,0) circle (.1) node[below]{$y_3$};
       \fill[fill=white,draw=black] (15,0) circle (.1) node[below]{$y_4$};
       \fill[fill=white,draw=black] (20,0) circle (.1) node[below]{$y_5$};
       \fill[fill=white,draw=black] (0,5) circle (.1) node[above]{$x_1$};
       \fill[fill=white,draw=black] (5,5) circle (.1) node[above]{$x_2$};
        \fill[fill=white,draw=black] (10,5) circle (.1) node[above]{$x_3$};
        \fill[fill=white,draw=black] (15,5) circle (.1) node[above]{$x_4$};
\end{tikzpicture}
    \caption{The graph $T_\lambda$ for $\lambda = (5,3,2,1)$}
    \label{fig_ferrers}
  \end{figure}
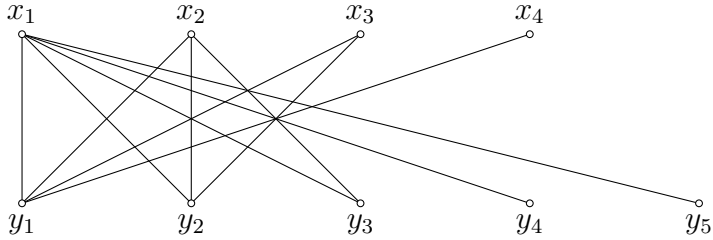

Next we consider the graphs studied in Galetto, 
{\it et al.} \cite{GHKKPVT} as our second family of graphs.  For integers $r \geq 3$ and $d \geq 2$,
we let $G_{r,d}$ be the graph with vertex set
$$V(G_{r,d}) = \{x_1,x_2,y_1,\ldots,y_d,z_1,\ldots,z_{2r-3}\}$$
and edge set
\begin{eqnarray*}
E(G_{r,d}) &= &\{\{x_i,y_j\} ~|~ 1 \leq i \leq 2,~~ 1 \leq j \leq d\} \cup \\
&  & \{\{x_1,z_1\},\{z_1,z_2\},\{z_2,z_3\},\ldots,\{z_{2r-4},z_{2r-3}\},
\{z_{2r-3},x_2\}\}.
\end{eqnarray*}
Observe that $G_{r,d}$ is the graph formed by taking
the complete bipartite graph $K_{2,d}$ (defined below), and then joining the
two vertices of degree $d$ by a path of length $2r-2$. 
As an example, see Figure \ref{fig_g35} for the graph $G_{3,5}$.
 \begin{figure}[!ht]
    \centering
    \begin{tikzpicture}[scale=0.45]
      \draw (0,0) -- (5,5);
      \draw (0,0) -- (15,5);
       \draw (5,0) -- (5,5);
      \draw (5,0) -- (15,5);
       \draw (10,0) -- (5,5);
      \draw (10,0) -- (15,5);
       \draw (15,0) -- (5,5);
      \draw (15,0) -- (15,5);
       \draw (20,0) -- (5,5);
      \draw (9,4.4) node{$a_5$};
      \draw (20,0) -- (15,5);
      \draw (2.2,3.5) node{$a_1$};
      \draw (2.4,1.5) node{$b_1$};
      \draw (4.3,3) node{$a_2$};
      \draw (6,3.1) node{$a_3$};
      \draw (8,3.2) node{$a_4$};
      \draw (6,1.2) node{$b_2$};
      \draw (11.6,.8) node{$b_3$};
      \draw (15.5,1) node{$b_4$};
      \draw (20,1)node{$b_5$};
      \draw (5,5) -- (5,8)node[midway, left] {$e_1$};
      \draw (5,8) -- (10,10)node[midway, above] {$e_2$};
      \draw (10,10) -- (15,8)node[midway, above] {$e_3$};
      \draw (15,8) -- (15,5)node[midway, right] {$e_4$};
   
      \fill[fill=white,draw=black] (0,0) circle (.1) node[below]{$y_1$};
       \fill[fill=white,draw=black] (5,0) circle (.1) node[below]{$y_2$};
       \fill[fill=white,draw=black] (10,0) circle (.1) node[below]{$y_3$};
       \fill[fill=white,draw=black] (15,0) circle (.1) node[below]{$y_4$};
       \fill[fill=white,draw=black] (20,0) circle (.1) node[below]{$y_5$};
       \fill[fill=white,draw=black] (5,5) circle (.1) node[left]{$x_1$};
       \fill[fill=white,draw=black] (15,5) circle (.1) node[right]{$x_2$};
       \fill[fill=white,draw=black] (5,8) circle (.1) node[left]{$z_1$};
       \fill[fill=white,draw=black] (15,8) circle (.1) node[right]{$z_3$};
       \fill[fill=white,draw=black] (10,10) circle (.1) node[above] {$z_2$};
     
    \end{tikzpicture}
    \caption{The graph $G_{3,5}$}
    \label{fig_g35}
 \end{figure}
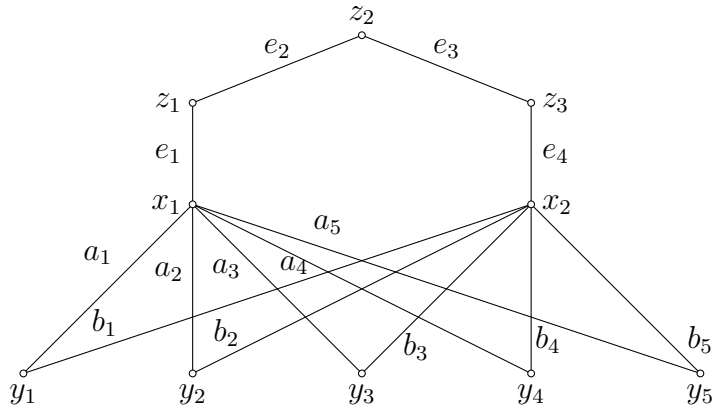
We label the edges so that $a_i = \{x_1,y_i\}$ and $b_i = \{x_2,y_i\}$ for $i=1,\ldots,d$, 
and $e_1 = \{x_1,z_1\}$, $e_{2r-2} = \{z_{2r-3},x_2\}$ and
$e_{i+1} = \{z_i,z_{i+1}\}$ for $1 \leq i \leq 2r-4$.

Using the above labelling, we can describe the universal
 Gr\"obner basis of $I_{G_{r,d}}$. 
 
 \begin{theorem}[{\cite[Corollary 3.3]{GHKKPVT}}]\label{universalGB}
 Fix  integers $r \geq 3$ and $d \geq 2$. A universal Gr\"obner  basis  for
 $I_{G_{r,d}}$ is given by
 $$\{a_ib_j - b_ia_j ~|~ 1\leq i < j\leq d\}\cup 
 \{a_ie_2e_4\cdots e_{2r-2} - b_ie_1e_3e_5 \cdots e_{2r-3} ~|~1\leq i \leq d \}.$$
 \end{theorem}
 
The next result provides many examples of toric ideals
 which are geometrically vertex decomposable.
In the statement below, 
the {\it complete bipartite graph}
$K_{n,m}$ is the graph with vertex
set $V = \{x_1,\ldots,x_n,y_1,\ldots,y_m\}$ and edge set $\{\{x_i,y_j\} ~|~ 1 \leq i \leq n,~~ 1 \leq j \leq m \}$.

 \begin{theorem}\label{families}
 The toric ideals of the following
  families of graphs are geometrically
 vertex decomposable:
 \begin{enumerate}
    \item $G$ is a cycle;
    \item $G$ is a Ferrers graph $T_\lambda$
    for any partition $\lambda$;
    \item $G$ is a complete bipartite
    graph $K_{n,m}$; and 
    \item $G$ is the graph $G_{r,d}$
    for any $r\geq 3, d\geq 2$.
\end{enumerate}
\end{theorem}

\begin{proof} 
(1) Suppose that $G$ is a cycle with $2n$ edges.  Then $I_G = \langle e_1e_3\cdots e_{2n-1} - e_2e_4\cdots e_{2n} \rangle$, so the result follows from Lemma \ref{simplecases} (2).  If $G$
is an odd cycle, then $I_G = \langle 0 \rangle$, and so it is geometrically
vertex decomposable by definition.

\noindent
(2)  As shown in the proof of 
\cite[Proposition 5.1]{CN}, the toric
ideal of $T_\lambda$ is generated by
the $2 \times 2$ minors of a one-sided ladder.
The ideal generated by the $2 \times 2$ minors of a one-sided
ladder is an example of Schubert determinantal ideal (e.g. see \cite{KMY}).   The conclusion now follows
from \cite[Proposition 5.2]{KR} which showed
that all Schubert determinantal ideals
are geometrically vertex decomposable.\footnote{It is not necessary to use the connection to Schubert determinantal ideals. Indeed, it is known from the ladder determinantal ideal literature that (mixed) ladder determinantal ideals from (two-sided) ladders possess initial ideals which are Stanley-Reisner ideals of vertex decomposable simplicial complexes (see \cite{GMN} and references therein). Then, an analogous proof to our proof of Theorem \ref{thm: gvdBipartite} can be given to show that these ideals are geometrically vertex decomposable.}

\noindent
(3) Apply the previous result using the partition
$\lambda = \underbrace{(m,m,\ldots,m)}_n$.

\noindent
(4)
Let $I = I_{G_{r,d}}$. Since it is a prime ideal, it is unmixed.
We first show that the statement holds if
$d =2$ and for any $r \geq 3$.
Let $y = a_2$, and consider the lexicographic order on 
 $\mathbb{K}[E(G_{r,d})] = \mathbb{K}[a_1,a_2,b_1,b_2,e_1,\ldots, e_{2r-2}]$
 with 
 $a_2 > a_1 > b_2 > b_1 > e_{2r-2} >\cdots  > e_1.$
This monomial order is $y$-compatible.

By using the universal Gr\"obner basis of 
Theorem \ref{universalGB}, we have
\begin{eqnarray*}
C_{y,I} &=& \langle b_1,e_2e_4\cdots e_{2r-2},a_1e_2\cdots e_{2r-2}
- b_1e_1e_3\cdots e_{2r-3} \rangle =  \langle b_1,e_2e_4\cdots e_{2r-2} \rangle
\end{eqnarray*}
and $N_{y,I} =  \langle a_1e_2\cdots e_{2r-2} - b_1e_1 \cdots e_{2r-3} \rangle.$   Note
that each binomial in $\mathcal{U}(I)$ is doubly
square-free, so we can use
Lemma \ref{square-freey} to deduce that
$${\rm in}_{y}(I) = C_{y,I} \cap (N_{y,I} + \langle y \rangle)$$ is a geometric vertex decomposition.   
To complete
this case, note that $C_{y,I}$ is a monomial
complete intersection in $\mathbb{K}[a_1,b_1,b_2,e_1,\ldots,e_{2r-2}],$
so this ideal is geometrically 
vertex decomposable by Corollary \ref{monomialcor}.
The ideal  $N_{y,I}$ is a principal ideal
generated by $a_1e_2\cdots e_{2r-2} - b_1e_1 \cdots e_{2r-3}$, so it is geometrically vertex decomposable
by Lemma \ref{simplecases} (2).  So, for
all $r \geq 3$, the toric ideal $I_{G_{r,2}}$ is
geometrically vertex decomposable.

We proceed by induction on $d$. Assume $d > 2$ and let 
$r \geq 3$.  Let $y = a_d$, and consider the lexicographic order on 
 $\mathbb{K}[E(G_{r,d})] = \mathbb{K}[a_1,\ldots,a_d,b_1,\ldots,b_d,e_1,\ldots, e_{2r-2}]$
 with 
 $a_d > \cdots > a_1 > b_d > \cdots > b_1 > e_{2r-2} >\cdots  > e_1.$
This monomial order is $y$-compatible.

By again appealing to Theorem \ref{universalGB}, we have
\begin{eqnarray*}
C_{y,I} &= &\langle b_1,\ldots,b_{d-1},
e_{2}e_4\cdots e_{2r-2} \rangle +\langle a_ib_j - b_ia_j ~|~ 1\leq i < j\leq d-1 \rangle + \\
&& \langle a_ie_2e_4\cdots e_{2r-2} - b_ie_1e_3e_5 \cdots e_{2r-3} ~|~1\leq i \leq d-1 \rangle \\
& = &  \langle b_1,\ldots,b_{d-1},e_2e_4\ldots e_{2r-2} \rangle,
\end{eqnarray*}
where the last equality comes from removing redundant generators.
On the other hand, by Lemma \ref{linktoricidealgraph}, $N_{y,I} = I_K$ where     $K = G_{r,d} \setminus a_d$.  Note that in this graph, the
edge $b_d$ is a leaf, and consequently, 
$N_{y,I} = I_{G_{r,d-1}}$ since $K \setminus b_d = G_{r,d-1}$.
     
We can again use Lemma \ref{square-freey} to deduce that
$${\rm in}_{y}(I) = C_{y,I} \cap (N_{y,I} + \langle y \rangle)$$ is a geometric vertex decomposition. 

To complete the proof, note that in the ring $\mathbb{K}[a_1,\ldots,a_{d-1},b_1,\ldots,b_d,e_1,\ldots,
e_{2r-2}]$, the ideal $C_{y,I}$ is geometrically vertex decomposable by Corollary \ref{monomialcor} since this ideal is a complete intersection monomial ideal. Also, the ideal $N_{y,I} = I_{G_{r,d-1}}$ is geometrically vertex decomposable by induction. Thus, $I_{G_{r,d}}$ is geometrically vertex decomposable for all $d \geq 2$ and $r \geq 3$.
\end{proof}

As we will see in the remainder of the paper, there are many non-bipartite graphs which have geometrically vertex decomposable toric ideals.

\section{Toric ideals with a square-free degeneration}\label{section_square-free}

As mentioned in the introduction, an important question in liaison
theory asks if every arithmetically Cohen-Macaulay subscheme of $\mathbb{P}^n$ is 
glicci (e.g. see \cite[Question 1.6]{KMMNP}).  
As shown by Klein and
Rajchgot (see Theorem \ref{gvd=>glicci}), if a homogeneous ideal
$I$ is a geometrically vertex decomposable ideal, then $I$ defines
an arithmetically Cohen-Macaulay subscheme, and furthermore, this scheme is glicci.  
It is therefore natural to ask if every toric ideal $I_G$
of a finite graph $G$ that has the property that $\mathbb{K}[E(G)]/I_G$ is Cohen-Macaulay is
also geometrically vertex decomposable.  If true, then
this would imply that the scheme defined by $I_G$ is glicci.  

Instead of considering all toric ideals of graphs such that $\mathbb{K}[E(G)]/I_G$ is Cohen-Macaulay, we restrict ourselves to ideals $I_G$ which possess a square-free Gr\"obner degeneration with respect to some monomial order $<$. 
By Theorem \ref{sqfree=>cm}, $\mathbb{K}[E(G)]/I_G$ is 
Cohen-Macaulay. Furthermore, if $\init_{<}(I_G)$ defines a vertex decomposable simplicial complex via the Stanley-Reisner correspondence, then $I_G$ would be geometrically 
vertex decomposable with respect to a \textit{lexicographic} monomial order $<$ (see \cite[Proposition 2.14]{KR}).  
We propose the conjecture below. Note that this conjecture
would imply that any toric ideal of a graph with a square-free initial ideal is 
glicci.  

\begin{conjecture}\label{mainconjecture}
Let $G$ be a finite simple graph with toric ideal $I_G \subseteq \mathbb{K}[E(G)]$.
If $\init_{<}(I_G)$ is square-free with respect to a lexicographic monomial order $<$, then $I_G$ is geometrically vertex decomposable.
\end{conjecture}

\noindent
By Theorem \ref{thm: gvdBipartite}, Conjecture \ref{mainconjecture} is true in the bipartite setting. In this section, we build a framework for proving Conjecture \ref{mainconjecture}. In particular, we reduce 
Conjecture \ref{mainconjecture} to checking whether certain related ideals are equidimensional, and we prove Conjecture \ref{mainconjecture} for the case where the generators in the universal Gr\"obner basis $\mathcal{U}(I_G)$ are quadratic. 

\subsection{Framework for the conjecture}
Suppose that $G$ is a labelled graph with $n$ edges $e_1,\dots, e_n$ and toric ideal $I_G\subseteq \mathbb{K}[E(G)]$. Let $<_G$ be the lexicographic monomial order induced from the ordering of the edges coming from the labelling.  That is, $e_1>e_2>\cdots >e_n$. 

We define a class of ideals of the form $I^G_{E,F}$ such that $E\cup F=E_k=\{e_1,\ldots,e_k\}$ for some $0\leq k \leq n$ with
$E \cap F = \emptyset$. Here $E_0 = \emptyset$. Define 
\[I^{G}_{E,F} := I_{G\setminus (E\cup F)} + M^{G}_{E,F}\] 
where $I_{G\setminus (E \cup F)}$ is the toric ideal of the graph $G$ with
the edges $E\cup F$ removed, and 
where $M^G_{E,F}$ is the ideal of $\mathbb{K}[e_1,\dots, e_n]$ generated by those monomials $m$ with $m\ell -p\in \mathcal{U}(I_G)$ such that:

\begin{enumerate}
    \item $\init_{<_G}(m\ell -p) = m\ell$,
    \item $\ell$ is a monomial only involving some non-empty subset of variables in $E$, and
    \item no $f\in F$ divides $m\ell$ and no $e\in E$ divides $m$.
\end{enumerate}
If there are no monomials $m$ which satisfy conditions (1), (2), and (3), we set $M^G_{E,F} = \langle 0 \rangle$.  
Therefore $M^G_{\emptyset,F} =\langle 0\rangle$ and $I^G_{\emptyset,F} = I_{G\setminus F}$ (which is generated by those primitive closed even walks in $G$ which do not pass through any edge of $F=E_k$). On the other hand, if there is an
$\ell -p \in \mathcal{U}(I_G)$ with ${\rm in}_{<_G}(\ell -p) = \ell$ 
where $\ell$ is a monomial only involving the variables in $E$, then we take
$m = 1$, and so $M_{E,F}^G = \langle 1 \rangle$.

There is a natural set of generators for $I^G_{E,F}$ using the primitive closed even walks of $I_G$. 
In particular, the ideal $I^G_{E,F}$ is generated by the
set
\[\mathcal{U}(I_{G\setminus (E\cup F)})\cup\mathcal{U}(M^G_{E,F}),\] 
where $\mathcal{U}(I_{G\setminus (F\cup E)})$ is the set of binomials defined by primitive closed even walks of the graph $G\setminus (E \cup F)$, and $\mathcal{U}(M^G_{E,F})$ are those monomials $m$ appearing in a generator of $\mathcal{U}(I_G)$ and satisfying conditions (1), (2), and (3) above.   Because $M^G_{E,F}$ is a monomial ideal,
its minimal generators form a universal Gr\"obner
basis, so our notation makes sense.
Going forward, we restrict our attention to the case where $\init_{<_G}(I_G)$ is square-free (this setting includes families of graphs like gap-free graphs \cite{DAli} for certain choices of $<_G$). 

To illustrate some of the above ideas, we consider the 
case that $E \cup F = E_1 = \{e_1\}$.  This example also
highlights a connection 
to the geometric vertex decomposition of
$I_G$ with respect to $e_1$.

\begin{example}\label{tree} Assume that $\init_{<_G}(I_G)$ is square-free. Then we can write \[\mathcal{U}(I_G) = \{e_1m_1-p_1,\ldots ,e_1m_r-p_r,t_1,\ldots, t_s\}\] where $e_1$ does not divide $m_i,p_i$ or any term of $t_i$. This set defines a universal Gr\"obner basis for $I_G = I^G_{\emptyset, \emptyset}$. Since $I_{G\setminus e_1}=\langle t_1,\ldots, t_s \rangle$ (by Lemma~\ref{linktoricidealgraph}), we can write 
\begin{align*}
\init_{e_1}(I^G_{\emptyset,\emptyset}) 
&=\langle e_1m_1,\ldots, e_1m_r, t_1,\ldots, t_s \rangle\\
&= \langle e_1,t_1,\ldots, t_s \rangle\cap \langle m_1,\ldots, m_r, t_1,\ldots, t_s \rangle\\
&= (\langle e_1\rangle + I_{G\setminus e_1})\cap (M^G_{\{e_1\},\emptyset} + I_{G\setminus e_1})\\
&= (\langle e_1\rangle + I_{G\setminus e_1} + M^G_{\emptyset,\{e_1\}} )\cap I^G_{\{e_1\},\emptyset}\\
&= (\langle e_1\rangle + I^G_{\emptyset,\{e_1\}})\cap I^G_{\{e_1\},\emptyset}. 
\end{align*}
Note that $I_{G\setminus e_1}  = I_{G\setminus e_1}+M^G_{\emptyset, \{e_1\}}$ since $M^G_{\emptyset, \{e_1\}} = \langle 0 \rangle$. 

Note that if we take $y=e_1$ and $I=I^G_{\emptyset,\emptyset}$, then we get $C_{y,I} = I^G_{\{e_1\},\emptyset}$ and $N_{y,I} = I^G_{\emptyset,\{e_1\}}$. That is, $y=e_1$ defines a geometric vertex decomposition of $I_G$. Therefore, when $E \cup F = E_1=\{e_1\}$, 
either $e_1\in E$ or $e_1\in F$, and each case appears in the geometric vertex decomposition. \qed
\end{example}

If we continue the process by taking $\init_{e_2}(\cdot)$ of $I^G_{\{e_1\},\emptyset}$ and of $I^G_{\emptyset,\{e_1\}}$, we get one of four possible $C_{y,I}$ and $N_{y,I}$ ideals, each corresponding to a possible distribution of $\{e_1,e_2\}$ into the disjoint sets $E$ and $F$ such that $E\cup F=E_2$.  Figure \ref{relationofideals} shows the relationship
between the ideals $I_{E,F}^G$ for the cases $E\cup F = E_i$ for $i=0,\ldots,3$.

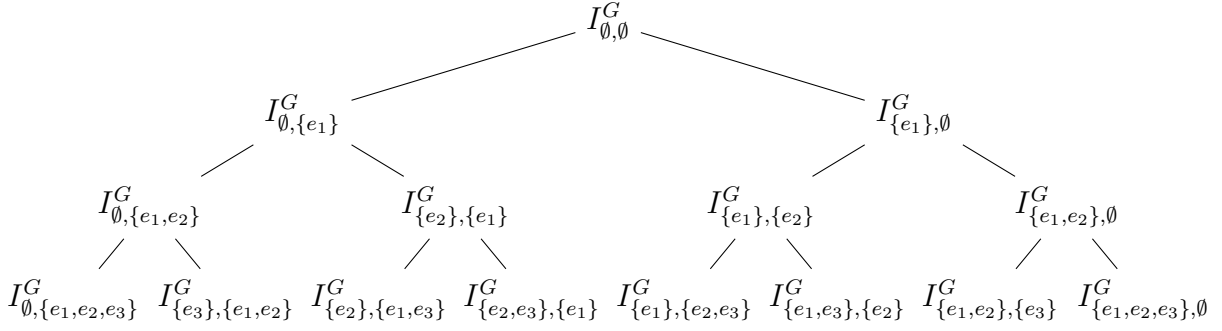
\begin{figure}
 \begin{tikzpicture}[sibling distance=24em, scale=0.82]
  \node {$I^G_{\emptyset, \emptyset}$}
    child { node {$I^G_{\emptyset,\{e_1\}}$}[sibling distance=12em]
      child { node {$I^G_{\emptyset, \{e_1,e_2\}}$}[sibling distance=6em]
        child {node {$I^G_{\emptyset, \{e_1,e_2,e_3\}}$}}
        child {node {$I^G_{\{e_3\}, \{e_1,e_2\}}$}}}
      child { node {$I^G_{\{e_2\}, \{e_1\}}$}[sibling distance=6em]
        child {node {$I^G_{\{e_2\}, \{e_1,e_3\}}$}}
        child {node {$I^G_{\{e_2,e_3\}, \{e_1\}}$}}}}
     child { node {$I^G_{\{e_1\},\emptyset}$}[sibling distance=12em]
      child { node {$I^G_{\{e_1\},\{e_2\}}$}[sibling distance=6em]
        child {node {$I^G_{\{e_1\},\{e_2,e_3\}}$}}
        child {node {$I^G_{\{e_1,e_3\},\{e_2\}}$}}}
      child { node {$I^G_{\{e_1,e_2\},\emptyset}$}[sibling distance=6em]
        child {node {$I^G_{\{e_1,e_2\},\{e_3\}}$}}
        child {node {$I^G_{\{e_1,e_2,e_3\},\emptyset}$}}}};
\end{tikzpicture}
\caption{The relation between the ideals $I^G_{E,F}$}
    \label{relationofideals}
\end{figure}

One strategy to verify Conjecture \ref{mainconjecture} is
to prove the following three statements:

\begin{itemize}
    \item[$(A)$] Given $I=I^G_{E,F}$ such that $E\cup F=E_{k-1}$ and $I\neq \langle 0\rangle$ or $\langle 1\rangle$, then $y=e_k$ defines a geometric vertex decomposition. Furthermore, $N_{y,I}$ and $C_{y,I}$ must also be of the form $I^G_{E',F'}$ where $E'\cup F' = E_k$.
    \item[$(B)$] If $E \cup F = E_n$, then $I^{G}_{E,F} = \langle 0 \rangle$ or $\langle 1 \rangle$.
    \item[$(C)$] For any $E\cup F=E_k$, the ideal $I^G_{E,F}$ must be unmixed.
\end{itemize}
Indeed, the next theorem verifies that proving
$(A), (B)$, and $(C)$ suffices to show that $I_G$
is geometrically vertex decomposable.

\begin{theorem}\label{thingstocheck}
Let $G$ be a finite simple graph with toric ideal $I_G \subseteq \mathbb{K}[E(G)]$, and suppose that $\init_{<}(I_G)$ is square-free with respect to a lexicographic monomial order $<$.
If statements $(A)$, $(B)$, and $(C)$ are true, then $I_G$ is geometrically vertex decomposable.
\end{theorem}

\begin{proof}
Let $n$ be the number of edges of $G$.   We  show that 
for all sets $E$ and $F$ such that  $E \cup F = E_k$, the ideal $I_{E,F}^G$ is geometrically vertex decomposable, and in particular, $I_{\emptyset,\emptyset}^G = I_G$ is geometrically vertex decomposable.  We do descending
induction on $|E \cup F|$.   If $|E \cup F| = n$, then
$E \cup F = E_n$, and so by statement $(B),$ $I^G_{E,F} = \langle 0 \rangle$ or $\langle 1 \rangle$, both of which are 
geometrically vertex decomposable by definition.

For the induction step, assume that all ideals of the form
$I^G_{E,F}$ with $E \cup F = E_\ell$ with $\ell \in \{k,\ldots,n\}$ are geometrically vertex decomposable.  
Suppose that $E$ and $F$ are two sets such that $E \cup F = E_{k-1}$.  The ideal $I^G_{E,F}$ 
is unmixed by 
statement $(C)$. If $I^G_{E,F}$ is $\langle 0\rangle$ or $\langle 1\rangle$, then it is geometrically vertex decomposable by definition. Otherwise, by statement $(A)$, 
the variable $y=e_k$ defines a geometric
vertex decomposition of $I =I_{E,F}^G$, i.e.,
$${\rm in}_{y}(I_{E,F}^G) = C_{y,I} \cap (N_{y,I} + \langle y \rangle).$$
Moreover, also by statement $(A)$, the ideals
$C_{y,I}$ and $N_{y,I}$ have the form $I^G_{E',F'}$ with
$E' \cup F'= E_k$.  By induction, these two ideals
are geometrically vertex decomposable.  So, $I^G_{E,F}$ is 
geometrically vertex decomposable.
\end{proof}

We now show that $(A)$ and $(B)$ are always true.  Thus,
to prove Conjecture \ref{mainconjecture}, one needs
to verify $(C)$.  In fact, we will show that it is enough
to show that $\mathbb{K}[E(G)]/I_{E,F}$ is equidimensional for
all ideals of the form $I^G_{E,F}$.

We begin by proving that statement
$(A)$ holds if ${\rm in}_{<_G}(I_G)$ is a square-free
monomial ideal.  In fact, we prove some additional
properties about the ideals $I^G_{E,F}$.

\begin{theorem}\label{CN_Grobner}
Let $I_G$ be the toric ideal of a finite simple graph $G$ such that $\init_{<_G}(I_G)$ is square-free.   
For each $k \in \{1,\ldots,n\}$, let
$E,F$ be  disjoint subsets of $\{e_1,\ldots,e_n\}$
such that $E \cup F  = E_{k-1} = \{e_1,\ldots,e_{k-1}\}$.
Then 
\begin{enumerate}
    \item The natural generators $\mathcal{U}(I_{G\setminus (E \cup F)}) \cup \mathcal{U}(M^G_{E,F})$ of $I^G_{E,F}$ form
    a Gr\"obner basis for $I^G_{E,F}$ with respect to
    $<_G$.   Furthermore, ${\rm in}_{<_G}(I^G_{E,F})$ is a square-free monomial ideal.
    \item $I^G_{E,F}$ is a radical ideal.
    \item The variable $y=e_k$ defines a geometric
    vertex decomposition of $I^{G}_{E,F}$.
    \item  If $I = I^{G}_{E,F}$ and $y=e_k$, then
    $C_{y,I} = I^G_{E \cup \{e_k\},F}$ and $N_{y,I} =
    I^G_{E,F\cup\{e_k\}}$;  in particular,
     $$ \init_{e_k}(I^G_{E,F}) = I^G_{E\cup\{e_k\},F}\cap ( I^G_{E,F\cup \{e_k\}} +\langle e_k \rangle).$$
\end{enumerate}
\end{theorem}

\begin{proof}
$(1)$
We will proceed by induction on $|E\cup F| = r = k-1$. If $r=0$,
then $E \cup F = \emptyset$ and $I^G_{E,F} = I_G$.  In
this case the natural generators are $\mathcal{U}(I_G) \cup \mathcal{U}(M^G_{\emptyset,\emptyset}) = \mathcal{U}(I_G)$,
and this set defines a universal Gr\"obner basis consisting of primitive closed even walks of $G$. Its initial ideal is square-free by the assumption on $<_G$. 

Now suppose that $|E \cup F| = r \geq 1$ and assume the result holds for $r-1$.  There are two cases to consider:
\vspace{.1cm}

\noindent\underline{Case 1}: Assume that $e_r\in E$. By induction, the natural generators
$$\mathcal{U}(I_{G\setminus({(E\setminus\{e_r\}) \cup F})}) \cup
\mathcal{U}(M^G_{E\setminus\{e_r\},F})$$ of 
$I^G_{E\setminus \{e_r\},F}$ is a Gr\"obner basis with
respect to $<_G$ and has a square-free initial ideal with respect to $<_G$. For the computations that follow, we can restrict to a minimal Gr\"obner basis by removing elements of this generating set  which do not 
have a square-free lead term. 

Since $e_r$ cannot divide both terms of a binomial defined by a primitive closed even walk, we must have that this minimal Gr\"obner basis is square-free in $y=e_r$ (any $e_r$ that appears in a binomial 
must appear in the lead term by definition of $<_G$, because none of the generators of $I^G_{E\setminus \{e_r\},F}$ involve $e_1,\ldots,e_{r-1}$). Therefore, $I^G_{E\setminus \{e_r\},F}$ has a geometric vertex decomposition with respect to $y$ by Lemma \ref{square-freey} (2). 

The ideal $C_{y,I^G_{E\setminus \{e_r\},F}}$ is therefore generated by:

\begin{itemize}
    \item Binomials corresponding to primitive closed even walks not passing through any edge of $E_r$. That is, elements of $\mathcal{U}(I_{G\setminus E_r})$.
   \item Monomials $m$ which appear as the coefficient of $e_r$ in $me_r -p\in \mathcal{U}(I_{G\setminus E_{r-1}})$.
    \item Monomials $m$ which appear as the coefficient of $e_r$ in $\mathcal{U}(M^G_{E\setminus \{e_r\},F})$. In this case, $m$ is part of a binomial $me_r\prod\limits_{i\in \mathcal{I}} e_i-p \in \mathcal{U}(I_G)$, where $\mathcal{I}$ 
    indexes a subset of $E\setminus \{e_r\}$.
\end{itemize}

\noindent The last two types of monomials are exactly those monomials defining $\mathcal{U}(M^G_{E,F})$. Therefore 
\[C_{y,I^G_{E\setminus \{e_r\},F}} = I^G_{E,F}.\]

\noindent Furthermore, the generators listed above for $C_{y,I^G_{E\setminus \{e_r\},F}}$ are a Gr\"obner basis with respect to $<_G$ by Lemma \ref{square-freey} (1) and are a subset of the natural generators 
of $I^G_{E,F}$. Its initial ideal is also square-free since we restricted to a minimal 
Gr\"obner basis before computing $C_{y,I^G_{E\setminus \{e_r\},F}}$.
\vspace{.1cm}

\noindent\underline{Case 2}: Assume that $e_r\in F$. We argue similarly to Case 1 and omit the details. By induction $\mathcal{U}(I^G_{E,F\setminus \{e_r\}})$ is a Gr\"obner basis with respect to $<_G$ and defines a square-free initial ideal. We can once again restrict to a minimal Gr\"obner basis, both ensuring that all lead terms are square-free and that $y=e_r$ defines a geometric vertex decomposition. In this case, $N_{y,I^G_{E,F\setminus \{e_r\}}} = I^G_{E,F}$, and $\mathcal{U}(I^G_{E,F\setminus \{e_r\}})$ is a Gr\"obner basis by Lemma \ref{square-freey} (1) with respect to $<_G$.   As in Case 1,
the initial ideal of $I^G_{E,F}$ is square-free with respect to this 
monomial order since we restricted to a minimal Gr\"obner basis when
computing $N_{y,I^G_{E,F\setminus\{e_r\}}}$.

For statement $(2)$, the ideal $I^G_{E,F}$ is radical because it has a square-free degeneration.  Statements $(3)$ and $(4)$ were shown as 
part of the proof of statement $(1)$.
\end{proof}

We now verify that statement $(B)$ holds.
\begin{theorem} \label{statementb}
Let $I_G$ be the toric ideal of a finite simple graph $G$ such that $\init_{<_G}(I_G)$ is square-free. If $E \cup F = E_n$, then $I^G_{E,F} = \langle 0 \rangle$ or 
$\langle 1 \rangle$.
\end{theorem}

\begin{proof}
Let $\mathcal{U}(I_G)$ be the universal Gr\"obner basis of $I_G$ defined in Theorem~\ref{generatordescription}.  Since
$\init_{<_G}(I_G)$ is square-free, we can take a minimal Gr\"obner basis where
each lead term is square-free.  We can write each element in our Gr\"obner basis
as a binomial of the form $m\ell -p$ with ${\rm in}_{<_G}(m\ell-p) = m\ell$ where
$\ell$ is a monomial only in the variables in $E$.  Suppose that
there is a binomial $m\ell -p \in \mathcal{U}(I_G)$ such that $m\ell = \ell$, i.e.,
the lead term only involves variables in $E$.  Then
$1 \in M^G_{E,F}$, and so $I^G_{E,F} = \langle 1 \rangle$, since the monomials
of $M^G_{E,F}$ form part of the generating set of $I^G_{E,F}$. 

Otherwise, for every $m\ell -p \in \mathcal{U}(I_G)$, there is a variable $e_j \not\in E$ such that
$e_j|m$.  Since $E \cup F = E_n$, we must have $e_j \in F$.  But then $m$
is not in $M^G_{E,F}$ since it fails to satisfy condition $(3)$ of being
a monomial in $M^G_{E,F}$, and thus $M^G_{E,F} = \langle 0 \rangle$.  Since
$G \setminus (E\cup F)$ is the graph $G$ with all of its
edges removed, $I_{G\setminus (E\cup F)} = \langle 0 \rangle$.  Thus
$I^G_{E,F} = \langle 0 \rangle$.
\end{proof}

To prove Conjecture \ref{mainconjecture}, it remains to verify
statement $(C)$; that is, each ideal $I^G_{E,F}$ must be  unmixed. This has proven difficult to show in general without specific restrictions on $G$. Nonetheless, the framework presented above leads to the next theorem 
which reduces statement
$(C)$ to showing that $\mathbb{K}[E(G)]/I^{G}_{E,F}$ is equidimensional.  Recall that a ring $R/I$ is {\it equidimensional} if  $\dim(R/I) = \dim(R/P)$ for all minimal primes $P$ of ${\rm Ass}_R(R/I)$.

\begin{theorem}\label{framework}
Let $I_G$ be the toric ideal of a finite simple graph $G$ such that $\init_{<_G}(I_G)$ is square-free.
If $\mathbb{K}[E(G)]/I^G_{E,F}$ is equidimensional for every choice of $E,F,\ell$ such that $E\cup F=E_\ell$ and $0\leq \ell \leq n$, then $I_G$ is geometrically
vertex decomposable.
\end{theorem}

\begin{proof}
In light of Theorems \ref{thingstocheck}, \ref{CN_Grobner}, and \ref{statementb}, we only need to check that each $I^G_{E,F}$ is unmixed. However, by Theorem ~\ref{CN_Grobner} (3), each ideal $I^G_{E,F}$ is radical, so being unmixed is equivalent to being equidimensional.
\end{proof}

\begin{remark}
The definition of $I^G_{E,F}$ is an extension of the setup of Constantinescu and Gorla in \cite{CG} and is also used in Section 5. It is designed to utilize known results about geometric vertex decomposition. In \cite{CG}, $G$ is a bipartite graph, and techniques from liaison theory are employed to prove that $I_G$ is glicci. Using a similar argument for general $G$, we can use  \[\init_{<_G}(I^G_{E,F})= e_k\init_{<_G}(I^G_{E\cup\{e_k\},F}) + \init_{<_G}(I^G_{E,F\cup \{e_k\}}) \] to show that $\init_{<_G}(I^G_{E,F})$ can be obtained from $\init_{<_G}(I^G_{E\cup\{e_k\},F})$ via a Basic Double G-link  (see \cite[Lemma 2.1 and Theorem 2.8]{CG}), and so $\init_{<_G}(I^G_{E,F})$ being Cohen-Macaulay implies that $\init_{<_G}(I^G_{E\cup\{e_k\},F})$ is too (see Lemma \ref{linkage_CM}). Through induction, we could then prove that some (but not all) of the $I^G_{E,F}$ in the tree following Example~\ref{tree} are Cohen-Macaulay. 

On the other hand, to produce $G$-biliaisons as in \cite[Theorem 2.11]{CG}, we would need specialized information about the graph $G$, something which is not a straightforward extension of the bipartite case. 
\end{remark}

\subsection{Proof of the conjecture in the quadratic case}\label{quadratic}
In the case that $\mathcal{U}(I_G)$ contains only quadratic binomials, 
we are able to verify that Conjecture \ref{mainconjecture} is true,
that is, $I_G$ is geometrically vertex decomposable.  We first
show that when $\mathcal{U}(I_G)$ contains only quadratic binomials, it
has the property that ${\rm in}_{<_G}(I_G)$ is a square-free
monomial ideal for any monomial order.  In the statement
below, recall that a binomial $m_1-m_2$ is doubly square-free if both monomials
that make up the binomial are square-free.

\begin{lemma}\label{quadratic_square-free}
Suppose that $G$ is a graph such that $I_G$ has a universal
Gr\"obner basis $\mathcal{U}(I_G)$ of quadratic binomials.  
Then these generators are doubly square-free.
\end{lemma}

\begin{proof}
By Theorem \ref{generatordescription}, 
a quadratic element of $\mathcal{U}(I_G)$ comes from a primitive closed walk of length 
four of $G$. Since consecutive edges cannot be equal, all primitive walks of length four are actually cycles, so no edge is repeated,
or equivalently, the generator is doubly square-free.
\end{proof}

As noted in the previous subsection, to verify the conjecture
in this case, it suffices to show that $\mathbb{K}[E(G)]/I^G_{E,F}$ is equidimensional for all $E,F,\ell$ with $E \cup F = E_\ell$.  In fact,
we will show a stronger result and show that all of these rings are 
Cohen-Macaulay.

We start with the useful observation that the natural set of generators of $I^G_{E,F}$ actually defines a universal Gr\"obner basis for the ideal.

\begin{lemma}\label{universal_GEF}
Under the assumptions of Theorem~\ref{CN_Grobner}, $\mathcal{U}(I_{G\setminus E_\ell})\cup\mathcal{U}(M^G_{E,F})$ is a universal Gr\"obner basis of $I^G_{E,F}$.
\end{lemma}
\begin{proof}
We will proceed by induction on $|E\cup F|$. The result is clear when $|E\cup F|=0$. For the induction step, observe that $I^G_{E,F}$ is either $N_{y,I^G_{E,F\setminus y}}$ or $C_{y,I^G_{E\setminus y,F}}$ for some variable $y=e_i$. Suppose towards a contradiction that there is some monomial order $<$ on $\mathbb{K}[e_1,\ldots,\hat{y},\ldots e_n]$ for which $\mathcal{U}(I^G_{E,F})$ is not a Gr\"obner basis. Extend $<$ to a monomial order $<_y$ on $\mathbb{K}[e_1,\ldots,e_n]$ which first chooses terms with the highest degree in $y$ and breaks ties using $<$. Clearly $<_y$ is a $y$-compatible order. By \cite[Theorem 2.1]{KMY}, $\mathcal{U}(I^G_{E,F})$ is a Gr\"obner basis with respect to $<_y$. But $<_y=<$ on $\mathbb{K}[e_1,\ldots,\hat{y},\ldots e_n]$, a contradiction.
\end{proof}

\begin{lemma}\label{CMspecialcase}
Let $R=\mathbb{K}[E(G)]$, and suppose that $G$ is finite simple graph such that $I_G$ has a universal Gr\"obner basis $\mathcal{U}(I_G)$ of quadratic binomials. Then $R/I^G_{E,F}$ is Cohen-Macaulay for every choice of $E,F$ and $\ell$ such that $E\cup F=E_\ell$.
\end{lemma}
\begin{proof}
Fix some $E$ and $F$ such that $E\cup F=E_\ell$. By definition $I^G_{E,F} = I_{G\setminus E_\ell} + M^G_{E,F}$. Since $\mathcal{U}(I_G)$ consists of quadratic binomials, then $M^G_{E,F}$ is either $\langle 1\rangle, \langle 0\rangle,$ or $\langle e_{i_1},\ldots,e_{i_s}\rangle$ with $s>0$.

The statement of the theorem clearly holds if $M^G_{E,F}=\langle 1\rangle$.  If $M^G_{E,F}=\langle 0\rangle$, then $I^G_{E,F} = I_{G\setminus E_\ell}$. Then $I_{G\setminus E_\ell}$ is generated by quadratic primitive binomials and therefore possesses a square-free degeneration. By Theorem \ref{sqfree=>cm} these are toric ideals of graphs that are Cohen-Macaulay. We are left with the case that $M^G_{E,F}$ is generated by $s$ indeterminates. 

We first show that each $I^G_{E,F}$ is actually equal 
to $\widetilde{I}^G_{E,F}:= I_{G\setminus (E_\ell\cup\{e_{i_1},\ldots, e_{i_s})\}}+M^G_{E,F}$.  We certainly have $\widetilde{I}^G_{E,F} \subset I^G_{E,F}$. Let $<_{E,F}$ be the monomial order $e_{i_1}>\cdots >e_{i_s}$ and $e_{i_s}>f$ for all $f\in E(G)\setminus (E_\ell\cup\{e_{i_1},\ldots, e_{i_s})\}$. By Lemma \ref{universal_GEF}, $\mathcal{U}(I_{G\setminus E_\ell})\cup\mathcal{U}(M^G_{E,F})$ is a universal Gr\"obner basis for $I^G_{E,F}$. A similar statement holds for $\widetilde{I}^G_{E,F}$ since no variable of $\mathcal{U}(M^G_{E,F})$ is used to define $I_{G\setminus (E_\ell\cup\{e_{i_1},\ldots, e_{i_s})\}}$.

Clearly $\init_{<_{E,F}}(\widetilde{I}^G_{E,F}) \subset \init_{<_{E,F}}(I^G_{E,F})$. On the other hand, if there is some $u-v\in\mathcal{U}(I_{G\setminus E_\ell})$ where $u$ or $v$ is in the ideal $M^G_{E,F}$, then $\init_{<_{E,F}}(u-v)$ is a multiple of some $e_{i_j}$ for $j\in\{1,\ldots,s\}$. Therefore, $\init_{<_{E,F}}(\widetilde{I}^G_{E,F}) = \init_{<_{E,F}}(I^G_{E,F})$ which in turn implies that $\widetilde{I}^G_{E,F} =  I^G_{E,F}$ (e.g. see \cite[Problem 2.8]{EH}).

Therefore, we can show that $R/I^G_{E,F}$ is Cohen-Macaulay by proving that $R/\widetilde{I}^G_{E,F}$ is. Recall that if a ring $S$ is Cohen-Macaulay and graded and $x$ is a non-zero-divisor of $S$, then $S/\langle x \rangle$ is also Cohen-Macaulay.

Now it is easy to see that $e_{i_1} + I_{G\setminus (E_\ell\cup\{e_{i_1},\ldots, e_{i_s})\}},\ldots, e_{i_s} + I_{G\setminus (E_\ell\cup\{e_{i_1},\ldots, e_{i_s})\}}$
is a regular sequence on $R/I_{G\setminus (E_\ell\cup\{e_{i_1},\ldots, e_{i_s})\}}$. This follows from the fact that $I_{G\setminus (E_\ell\cup\{e_{i_1},\ldots, e_{i_s})\}}$
is Cohen-Macaulay since it possesses a square-free degeneration, and from the fact that $\mathcal{U}(I_{G\setminus (E_\ell\cup\{e_{i_1},\ldots, e_{i_s})\}})$ is not 
defined using the variables $\{e_{i_1},\ldots,e_{i_s}\}$.
\end{proof}

The previous lemma provides the unmixed condition needed to use  Theorem~\ref{framework}.  In summary, we have the following
result:

\begin{theorem}\label{quadratic_GVD}
Let $I_G$ be the toric ideal of a finite simple graph $G$ 
such that $\mathcal{U}(I_G)$ consists of quadratic binomials. Then $I_G$ is geometrically vertex decomposable and glicci.
\end{theorem}
\begin{proof}
By Lemma~\ref{quadratic_square-free}, any lexicographic order on the variables will determine a square-free degeneration of $I_G$. By Lemma~\ref{CMspecialcase} the rings $\mathbb{K}[E(G)]/I^G_{E,F}$ are Cohen-Macaulay for
all $E,F,$ and $\ell$ such that $E \cup F = E_\ell$. 
In particular, all of these rings are equidimensional.  Thus,
by Theorem~\ref{framework}, $I_G$ is geometrically vertex decomposable, and therefore glicci by Theorem \ref{gvd=>glicci}.
\end{proof}

\begin{remark}
Although the condition that $\mathcal{U}(I_G)$ consists of quadratic binomials is restrictive, it is worth noting that there are families of graphs for which this is true (e.g. certain bipartite graphs). See \cite[Theorem 1.2]{OH} for a characterization of when $I_G$ can be generated by quadratic binomials, and \cite[Proposition 1.3]{HNOS} for the case where the Gr\"obner basis is quadratic.
\end{remark}

\end{document}